\documentclass[a4paper,12pt,onecolumn]{gf910art} 			

\usepackage{enumerate}
\usepackage{amsmath}
\usepackage{amsfonts}
\usepackage{amssymb}
\usepackage{fancyhdr}
\usepackage{bm}
\usepackage{tikz}
\usetikzlibrary{patterns}
\usepackage{multirow}
\usepackage[T1]{fontenc}
\usepackage{chngcntr}
\usepackage{apptools}
\AtAppendix{\counterwithin{lemma}{section}}

\newcommand{\TheTitle}{Optimization with affine homogeneous quadratic integral inequality constraints\footnote{This work was partly supported by the Engineering and Physical Sciences Research Council (EPSRC) grant EP/J010537/1. Email addresses: 
\email{gf910@ic.ac.uk} (Giovanni Fantuzzi), \email{a.wynn@imperial.ac.uk} 
(Andrew Wynn), \email{paul.goulart@eng.ox.ac.uk} (Paul Goulart), 
\email{antonis@eng.ox.ac.uk} (Antonis Papachristodoulou).}}

\title{\center\Large\bf{\TheTitle}}

\author{
\institute{Department of Aeronautics, Imperial College London, London, UK.}
\institute{Department of Engineering Science, University of Oxford, Oxford, UK.}
Giovanni Fantuzzi\footnotemark[2] 
\and
Andrew Wynn\footnotemark[2] \and
Paul Goulart\footnotemark[3] \and
Antonis Papachristodoulou\footnotemark[3]
}

\renewcommand{\vec}[1]{\bm{#1}}
\newcommand{\F}{\vec{F}}
\newcommand{\at}[1]{(#1)}
\newcommand{\der}[1]{\partial^{#1}}
\newcommand{\Leg}{\mathcal{L}}

\newcommand{\norm}[1]{\left\| #1 \right\|}
\newcommand{\dx}{\,{\rm d}x }
\newcommand{\dy}{\,{\rm d}y }
\newcommand{\defeq}{:=}

\theoremstyle{definition}
\newtheorem{assumption}{Assumption} 			
\crefname{assumption}{Assumption}{Assumptions}	

\let\OLDthebibliography\thebibliography
\renewcommand\thebibliography[1]{
  \OLDthebibliography{#1}
  \footnotesize
  \setlength{\parskip}{1pt}
  \setlength{\itemsep}{1pt plus 0.5ex}
}

\begin{document}

\maketitle

\begin{abstract}
We introduce a new technique to optimize a linear cost function subject
to a one-dimensional affine homogeneous quadratic integral inequality, i.e., the
requirement that a homogeneous quadratic integral functional,
affine in the optimization variables, is non-negative over a space of
functions defined by homogeneous boundary conditions. Such problems arise in stability analysis, input-to-state/output analysis, and control of many systems governed by partial differential equations (PDEs), in particular fluid dynamical systems. First, we derive outer approximations for the feasible set of a homogeneous quadratic integral inequality in terms of linear matrix inequalities (LMIs), and show that under mild assumptions a convergent, non-decreasing sequence  of lower bounds for the optimal cost can be computed with a sequence of semidefinite programs (SDPs).  Second, we obtain inner approximations in terms of LMIs and sum-of-squares constraints, so upper bounds for the optimal cost and strictly feasible points for the integral inequality can also be computed with SDPs. To aid the formulation and solution of our SDP relaxations, we implement our techniques in \textsc{QUINOPT}, an open-source add-on to \textsc{YALMIP}. We demonstrate our techniques by solving problems arising from the stability analysis of PDEs.
\end{abstract}



\section{Introduction}
\label{S:Intro}

Analysis and control of
systems governed by partial differential equations (PDEs) are 
fundamental problems in physics and engineering, but are challenging because the system state is a (vector-valued) 
function $\vec{w}$ of both the time $t$ and the spatial position vector $\vec{x}$, and as such it belongs to an 
infinite-dimensional space (e.g. a Sobolev space). 

In an effort to reduce the conservativeness introduced by finite-dimensional approximations, recent years have seen the development of analytical techniques that consider directly the infinite-dimensional PDEs, and lead to consideration of integral inequalities.
For example, the stability of an equilibrium of a PDE system in a domain $\Omega$, or of a control policy designed to stabilize it,  can be established by constructing a positive integral Lyapunov functional 
$\mathcal{V}(t) = \mathcal{V}\{\vec{w}(t,\cdot)\} = \int_\Omega 
V[\vec{w}(t,\vec{x})] d\vec{x}$ whose time derivative (also an integral quantity) is 
non-positive~\cite{Straughan2004,Valmorbida2014,Valmorbida2016}.
Other input-to-state/output properties such as passivity, reachability, and input-to-state stability can be studied in a similar way using dissipation inequalities for integral functionals of the state 
variable~\cite{Ahmadi2014,Ahmadi2016}. 
Finally, the computational cost of designing optimal control policies for systems with complex dynamics, such as turbulent flows, may be reduced by requiring the control law to minimize an upper bound on the objective function rather than the objective itself~\cite{Huang2015,Huang2015a,Lasagna2016,Lasagna2016a}, and in the case of PDEs such upper bounds can be found by solving suitable integral inequalities~\cite{Constantin1995,Constantin1995a,Doering1992,Doering1994,Doering1996,Hagstrom2014}.

When the underlying PDE system is {\it autonomous}, the integral inequalities obtained in all aforementioned applications depend on time only through the state $\vec{w}(t,\vec{x})$, and since they are imposed {\it pointwise in time}, the time dependence of $\vec{w}$ can be dropped. Checking a certain integral inequality for given system parameters, or alternatively optimizing the system parameters while satisfying an integral inequality, then requires solving optimization problems of the form
\begin{equation}
\label{E:IntroProblem}
\begin{gathered}
\min_{\vec{\gamma}} \quad \vec{c}^T \vec{\gamma} \\
\text{s.t.} \quad \mathcal{F}_{\vec{\gamma}} \{\vec{w}\} \defeq 
\int_{\Omega} F_{\vec{\gamma}}(\vec{x},\mathcal{D}^{\vec{k}}\vec{w}) 
{\rm d}^n\vec{x}
\geq 0,\quad \vec{w} \in H,
\end{gathered}
\end{equation}
where $H$ is a suitable function space, e.g. the 
space of all $\vec{k}$-times differentiable functions from $\Omega \subseteq \mathbb{R}^n$ (typically $n = 3$ for physical systems) to $\mathbb{R}^q$ that satisfy a given set of boundary 
conditions (BCs). The optimization variable $\vec{\gamma}\in\mathbb{R}^s$ represents a vector of tunable system parameters, $\vec{c}\in\mathbb{R}^s$ is the cost vector,  $F_{\vec{\gamma}}(\cdot,\cdot)$ is a function that depends parametrically on $\vec{\gamma}$, and
$\mathcal{D}^{\vec{k}}\vec{w}=[w_1, \partial_{x_1} w_1, \partial_{x_2} w_1,\hdots,\partial^{k_1}_{x_n} w_1,\hdots,\partial^{k_q}_{x_n} w_q]^T$ lists all 
partial derivatives of the components of $\vec{w}$ up to the order specified by the multi-index $\vec{k}=[k_1, \hdots, k_q]$.

When the dependence on $\vec{\gamma}$ is at least affine and strong duality holds, problem~\eqref{E:IntroProblem} could be solved (in principle) by first computing the 
minimizer $\vec{w}^\star$ of $\mathcal{F}_{\vec{\gamma}}$ as a function of 
$\vec{\gamma}$ using the calculus of 
variations~\cite{Courant1953,Giaquinta2004}, and then minimizing the augmented 
Lagrangian $L(\vec{\gamma}) = \vec{c}^T\vec{\gamma} - 
\lambda\mathcal{F}_{\vec{\gamma}}\{\vec{w}^\star\}$, where the 
Lagrange multiplier $\lambda\geq 0$ is chosen to enforce the integral inequality. 
This strategy has been 
successfully applied to some problems in fluid dynamics (see e.g.~\cite{Doering1997,Wen2013,Wen2015a}), but it requires careful, 
problem-dependent computations. 
Alternatively, when the integrand $F_{\vec{\gamma}}(\cdot,\cdot)$ is linear with respect to $\mathcal{D}^{\vec{k}}\vec{w}$ and polynomial in $\vec{x}$,~\eqref{E:IntroProblem} can be transformed into a semidefinite program (SDP) using integration by parts and moment relaxation techniques~\cite{Bertsimas2006}.
More recently, it has been suggested that~\eqref{E:IntroProblem} can be recast as an SDP even when the integrand is 
polynomial in $\mathcal{D}^{\vec{k}}\vec{w}$~\cite{Papachristodoulou2006,Valmorbida2014,Valmorbida2016,Valmorbida2015}:
one relates the derivatives of the components of $\vec{w}$ using integration by parts and algebraic identities, and then requires that the polynomial integrand $F_{\vec{\gamma}}(\vec{x},\mathcal{D}^{\vec{k}}\vec{w})$ admits a sum-of-squares (SOS) decomposition over the domain of integration. However, scalability issues usually prevent the solution of problems of practical interest because---as our examples will demonstrate---high-degree SOS relaxations are normally needed to achieve accurate results.

This paper presents a new approach to solving a class of problems of 
type~\eqref{E:IntroProblem}. We consider homogeneous quadratic functionals $\mathcal{F}_{\vec{\gamma}}$ over a one-dimensional compact domain; in other words, we assume that $\vec{x} \in \Omega \equiv [a,b] \subset \mathbb{R}$ and that the integrand $F_{\vec{\gamma}}(\vec{x},\mathcal{D}^{\vec{k}}\vec{w})$ is a homogeneous quadratic polynomial with respect to $\mathcal{D}^{\vec{k}}\vec{w}$. Inequalities of this type arise in many fluid or thermal convection systems of practical interest (see e.g.~\cite{Straughan2004,Constantin1995a,Doering1994,Doering1996,Ahmadi2015}) and these are the main applications we have in mind. Our techniques, already partially introduced by some of the authors for particular problem instances~\cite{Fantuzzi2015,Fantuzzi2016}, rely on Legendre series expansions to formulate SDPs with better scaling properties than the SOS method of~\cite{Valmorbida2016}.  Our main contributions are:

\begin{enumerate}
\item{For the first time, we formulate convergent {\it outer} approximations of the feasible set of~\eqref{E:IntroProblem} described by linear matrix inequalities (LMIs), so lower bounds for the optimal cost can be computed using SDPs.  }
\item{We extend the method of~\cite{Fantuzzi2015,Fantuzzi2016} to derive SDP-representable {\it inner} approximations for the feasible set of~\eqref{E:IntroProblem} in the general setting, so upper bounds on the optimal value of can also be obtained with semidefinite programming.}
\item{We present \textsc{QUINOPT}, an add-on to \textsc{YALMIP}~\cite{Lofberg2004,Lofberg2009} to aid the formulation of the SDP relaxations outlined above, and use it to solve examples that demonstrate the advantages (and some limitations) of our approach compared to the SOS method of~\cite{Valmorbida2016}.}
\end{enumerate}

The rest of the paper is organized as follows. 
Section~\ref{S:ProblemDef} introduces the class of optimization problems studied in this work; as a motivating example, we consider the stability analysis of a fluid flow driven by a surface stress~\cite{Hagstrom2014}. 
We formulate outer SDP relaxations in \S\ref{S:OuterApprox}, and inner SDP relaxations in \S\ref{S:InnerApprox}.
We remove some simplifying assumptions and further extend our results in \S\ref{S:Extensions}. Section~\ref{S:Results} presents \textsc{QUINOPT} and numerical examples arising from the analysis of PDEs, and we comment on the scalability of our methods in \S\ref{S:scalability}. Finally, \S\ref{S:Conclusion} offers concluding remarks and perspectives for future developments.

\paragraph*{Notation.}

Vectors and matrices are denoted by boldface characters; in particular, $\vec{0}$ denotes the zero vector/matrix.
The usual Euclidean and $\ell^1$ norms of $\vec{v}\in\mathbb{R}^n$ are
$\|\vec{v}\|=(\sum_{i=1}^{n}\vert v_i\vert^2)^{1/2}$ and
$\|\vec{v}\|_1=\sum_{i=1}^{n}\vert v_i\vert$, respectively. 
Given a matrix $\vec{Q}\in\mathbb{R}^{n\times m}$, the Frobenius norm is defined as 
$\|\vec{Q}\|_F=(\sum_{i=1}^{n}\sum_{j=1}^{m}\vert Q_{ij}\vert^2)^{1/2}$.
The range and null space of $\vec{Q}$ are denoted by $\mathcal{R}(\vec{Q})$ and $\mathcal{N}(\vec{Q})$, respectively. We denote the space of $n\times n$ symmetric matrices  by $\mathbb{S}^n$, and indicate that $\vec{Q}\in \mathbb{S}^n$ is positive semidefinite with the notation $\vec{Q}\succeq 0$.

For a compact interval $[a,b]\subset\mathbb{R}$ and a positive integer $q$, $C^m([a,b],\mathbb{R}^q)$ is the space of 
$m$-times continuously differentiable functions with domain $[a,b]$ and values in $\mathbb{R}^q$; we also write $C^m([a,b])$ for $C^m([a,b],\mathbb{R})$. 
Given $u\in C^m([a,b])$, 
$\|u\|_2$ and $\|u\|_\infty$ denote the usual 
$L^2(a,b)$ and $L^\infty(a,b)$ norms, 
\begin{align*}
\|u\|_2&=\left[\int_{a}^{b}\vert u(x) \vert^2 \dx\right]^{1/2}, &
\|u\|_\infty&=\sup_{x\in[a,b]} \vert u(x) \vert.
\end{align*}

The set of non-negative integers is denoted by $\mathbb{N}$, and $\mathbb{N}^q$ is the set of multi-indices of the form $\vec{\alpha}=[\alpha_1,\,\hdots,\,\alpha_q]$. The length of the multi-index $\vec{\alpha}\in\mathbb{N}^q$ is $\vert \vec{\alpha}\vert = \alpha_1 + \cdots+\alpha_q$. 
Given $\vec{w}\in C^m([a,b],\mathbb{R}^q)$ and
$\vec{\alpha},\vec{\beta} \in \mathbb{N}^q$ with  
$\alpha_i\leq\beta_i\leq m$ for all $i\in\{1,\,\hdots,\,q\}$, we define $\vec{\beta}-\vec{\alpha} = [\beta_1-\alpha_1,\,\hdots,\,\beta_q - \alpha_q]\in\mathbb{N}^q$ and we list all multi-index derivatives of order between $\vec{\alpha}$ and $\vec{\beta}$ in the vector
\begin{equation}
\label{E:MultIndDerDef}
\mathcal{D}^{[\vec{\alpha},\vec{\beta}]} \vec{w} \defeq 
\left[
\partial^{\alpha_1} u_1,\, \hdots,\, \partial^{\beta_1} u_1,\,
\partial^{\alpha_2} u_2,\, \hdots,\, \partial^{\beta_2} u_2,\, \hdots, 
\partial^{\beta_q} u_q
\right]^T
\in \mathbb{R}^{q+\vert \vec{\beta} - \vec{\alpha}\vert}.
\end{equation}
We also collect all boundary values of such derivatives in the vector
\begin{equation}
\label{E:BoundaryVector}
\mathcal{B}^{[\vec{\alpha},\vec{\beta}]} \vec{w} \defeq \begin{bmatrix}
\mathcal{D}^{[\vec{\alpha},\vec{\beta}]} \vec{w}\at{a} \\ \mathcal{D}^{[\vec{\alpha},\vec{\beta}]} \vec{w}\at{b}
\end{bmatrix}
\in \mathbb{R}^{2(q+\vert \vec{\beta} - \vec{\alpha}\vert)}.
\end{equation}
To simplify the notation, when $\vec{\alpha} = \vec{0}$ we will write $\mathcal{D}^{\vec{\beta}} \vec{w}$ and $\mathcal{B}^{\vec{\beta}} \vec{w}$ instead of $\mathcal{B}^{[\vec{0},\vec{\beta}]} \vec{w}$ and $\mathcal{D}^{[\vec{0},\vec{\beta}]} \vec{w}$.

Finally, given two scalar functions $f$, $g$ of a scalar variable $N$, we write $f \sim g$ to indicate that $f$ and $g$ are asymptotically equivalent up to multiplication by a positive constant, that is, $\lim_{N\to\infty} f/g = c$ for some positive constant $c$.

\section{Optimization with affine homogeneous quadratic integral inequalities}
\label{S:ProblemDef}

Let $\vec{\gamma}\in\mathbb{R}^s$ be a vector of optimization variables, and consider two integers $m,\,q$ and two multi-indices $\vec{k}=[k_1,\,\hdots,\,k_q],\,\vec{l}=[l_1,\,\hdots,\,l_q]\in\mathbb{N}^q$ such that
\begin{subequations}
\begin{align}
\label{E:klCond1}
1 &\leq k_i \leq m-1, &&i\in\{1,\,\hdots,\,q\},\\
\label{E:klCond2}
k_i &\leq l_i \leq m, &&i\in\{1,\,\hdots,\,q\}.
\end{align}
\end{subequations}
Moreover, let $\vec{F}_0(x),\,\hdots,\,\vec{F}_s(x) \in \mathbb{S}^{q+\vert\vec{k}\vert}$ be symmetric matrices of polynomials of $x$ of degree at most $d_{F}$ and define
\begin{equation}
\label{E:FmatDef}
\vec{F}(x;\vec{\gamma}) \defeq  \vec{F}_0(x) + \sum_{i=1}^{s} \gamma_i \vec{F}_i(x),
\end{equation}
i.e., $\vec{F}(x;\vec{\gamma})$ is a symmetric matrix of polynomials of $x$ of degree at most $d_{F}$, 
the coefficients of which are affine in $\vec{\gamma}$. 

Throughout this paper, we consider linear optimization problems of type~\eqref{E:IntroProblem} subject to \emph{affine homogeneous quadratic integral inequalities}, i.e., problems of the form
\begin{gather}
\label{E:OptProblem}
\min_{\vec{\gamma}}\quad \vec{c}^T \vec{\gamma}	\\
\text{s.t.}\,\, \mathcal{F}_{\vec{\gamma}}\{\vec{w}\} 
\!\defeq\!
\int_{-1}^{1}\!\left(\mathcal{D}^{\vec{k}}\vec{w}\right)^T\!\vec{F}(x;\vec{\gamma}) \mathcal{D}^{\vec{k}}\vec{w} \dx \geq 0,
\, \vec{w} \in H,
\nonumber
\end{gather}
where $\vec{c}\in \mathbb{R}^s$ is the cost vector, $\vec{F}(x;\vec{\gamma})$ is as in~\eqref{E:FmatDef}, and
\begin{equation}
\label{E:HspaceDef}
H \defeq \left\{ \vec{w}\in C^m\left([-1,1],\mathbb{R}^q\right):\,\, \vec{A} \mathcal{B}^{\vec{l}}\vec{w} = \vec{0} \right\}
\end{equation}
is the space of $m$-times continuously differentiable functions satisfying the $p$ homogeneous BCs defined by the matrix $\vec{A}\in\mathbb{R}^{p\times2(q+\vert \vec{l}\vert)}$. 
There is no loss of generality in fixing the integration domain for the functional $\mathcal{F}_{\vec{\gamma}}$ to $[-1,1]$ because any compact interval $[a,b]$ can be mapped to it with a change of integration variable.
An affine homogeneous quadratic integral inequality represents a convex constraint on $\vec{\gamma}$, which makes~\eqref{E:OptProblem} a convex optimization problem.

\begin{remark}
For the sake of generality, we allow the space $H$ to be defined by derivatives of higher order than those appearing in $\mathcal{F}_{\vec{\gamma}}\{\vec{w}\}$ (this can always be achieved by adding zero columns to $\vec{A}$). In the applications we have in mind, i.e., problems arising from the study of autonomous PDEs, this is not uncommon:  $H$ encodes the BCs of the solution of a PDE, which might involve all derivatives up to the order of the PDE; $\mathcal{F}_{\vec{\gamma}}\{\vec{w}\}$, instead, is typically derived from a weak formulation of the PDE, after integrating some terms by parts.
\end{remark}

\begin{assumption}
To ease the exposition, we only consider two-dimensional functions $\vec{w}=[u,v]^T\in C^m([-1,1],\mathbb{R}^2)$. We also restrict the attention to the uniform  multi-indices $\vec{k}=[k,k]$ and $\vec{l}=[l,l]$, where $k$ and $l$ satisfy~\eqref{E:klCond1} and~\eqref{E:klCond2}. As will be discussed in \S\ref{S:Extensions}, however, all our results hold for the general case.
\end{assumption}

\subsection{Motivating Example}
\label{S:MotivatingExample}

\begin{figure}[b]
\centering
\begin{tikzpicture}[scale=1.25]
\draw[black, very thick] (-3,0) -- (3,0);
\draw[black, very thick, dashed] (-3.5,0) -- (3,0);
\draw[black, very thick, dashed] (3,0) -- (3.5,0);
\draw[black, very thick] (-3,1.5) -- (3,1.5);
\draw[black, very thick, dashed] (-3.5,1.5) -- (-3,1.5);
\draw[black, very thick, dashed] (3,1.5) -- (3.5,1.5);
\draw[black, thick, ->] (-3.5,.75) -- (3.5,.75) node[anchor=west] {\small$x$};
\draw[black, thick, ->] (-3,-0.3) -- (-3,1.75) node[anchor=south] {\small$y$};
\draw[black, very thick] (-1,1.65) -- (1,1.65) -- (0.75,1.9);
\node[anchor=south east] at (-3,0) {\small$-1$};
\node[anchor=north east] at (-3,1.5) {\small$1$};
\node[anchor=south] at (0,1.65) {\small$0.5\gamma$};
\fill[pattern=north east lines, pattern color=black] (-3,0) rectangle (3,-0.25);
\end{tikzpicture}
\caption{Sketch of the flow setup in our motivating example. The two-dimensional fluid layer extends to infinity along the $x$ direction, is bounded at $y=-1$ by a solid boundary and is driven at the surface ($y=1$) by a shear stress of non-dimensional magnitude $0.5\gamma$.}
\label{F:ShearFlowFig}
\end{figure}
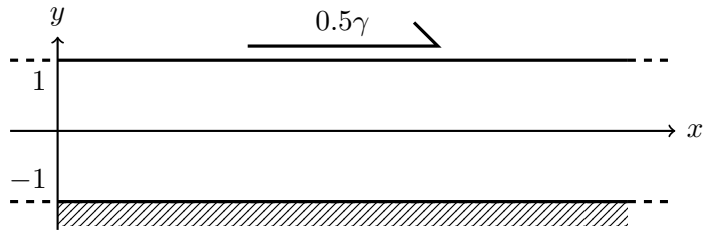

Consider a two-dimensional infinite layer of fluid bounded at $y=-1$ by a solid wall and driven at the surface at $y=1$ by a horizontal shear stress of non-dimensional magnitude $0.5\gamma$, as shown in Figure~\ref{F:ShearFlowFig}. The flow is governed by the incompressible Navier--Stokes equations, and admits a steady (i.e., time independent) solution in which the flow moves horizontally with velocity $\vec{w}_0 = (u_0,v_0)=(0.5\gamma y+0.5\gamma,0)$; see for example~\cite{Tang2004,Hagstrom2014,Fantuzzi2016}.  This steady flow is stable when the driving stress is small. The critical value $\gamma_\mathrm{cr}$ at which the steady flow is no longer guaranteed to be stable with respect to a sinusoidal perturbation $\vec{w}(y)e^{i\xi x + \sigma t}$ --- where $\vec{w}(y)=[u(y),v(y)]^T$ is the amplitude and $\xi$ is the wave number ---  is given by the solution of the optimization problem
\begin{equation}
\label{E:ProbEx2}
\begin{gathered}
\min \quad -\gamma\\
\begin{aligned}
\text{s.t.}\quad  \mathcal{F}_{\vec{\gamma}}\{\vec{w}\}\defeq
\int_{-1}^{1} 
&\left\{
\frac{16}{\xi^2}[ (\partial_y^2u)^2 + (\partial_y^2v)^2] + 8[(\partial_y u)^2 
\right. \\
&\,\left.
+ (\partial_y v)^2] 
+ \xi^2(u^2+v^2) + \frac{2\gamma}{\xi}(v\partial_y u - u\partial_y v)
\right\} \dy \geq 0,
\end{aligned}
\end{gathered}
\end{equation}
where the integral inequality constraint should hold for all functions $u,v\in C^2([-1,1])$ satisfying the homogeneous BCs
\begin{equation}
\label{E:Prob2BCs}
\begin{aligned}
u\at{-1} &= u\at{1} = \partial_y u\at{-1} = \partial_y^2 u\at{1}=0, \\
v\at{-1} &=  v\at{1} = \,\partial_y v\at{-1} = \partial_y^2 v\at{1}=0.
\end{aligned}
\end{equation}
See~\cite{Tang2004,Hagstrom2014} for a detailed discussion. The constraint in~\eqref{E:ProbEx2} can be rewritten in matrix form as in~\eqref{E:OptProblem} with $\vec{k}=\vec{l}=[2,2]$ and
\begin{align*}
\mathcal{D}^{\vec{k}} \vec{w} &= \begin{bmatrix}
u \\ \partial_y u \\ \partial_y^2 u \\  v \\ \partial_y v \\ \partial_y^2 v
\end{bmatrix},
&
\vec{F}(x;\vec{\gamma}) &= 
\begin{bmatrix}
\xi^2  & 0 & 0 & 0 & -\frac{\gamma}{\xi} & 0 \\
     0 & 8 & 0 & \frac{\gamma}{\xi} & 0 & 0 \\
     0 & 0 & \frac{16}{\xi^2}&  0 & 0 & 0 \\
     0 & \frac{\gamma}{\xi} & 0 & \xi^2  & 0 & 0 \\
     -\frac{\gamma}{\xi} & 0 & 0 &0 & 8 & 0 \\
     0 & 0 & 0 & 0 & 0 & \frac{16}{\xi^2}
\end{bmatrix}.
\end{align*}
Note that the matrix $\vec{F}$ above can be written in the form~\eqref{E:FmatDef} with $s=1$.
The reader can easily verify that the BCs on $u$ and $v$ can also be rewritten in the matrix form $\vec{A}\,\mathcal{B}^{\vec{l}}\vec{w} = \vec{0}$ with $\vec{A}\in \mathbb{R}^{8\times 12}$; we omit the details for brevity.
For this problem, it is clear that $\mathcal{F}_{\vec{\gamma}}\{\vec{w}\}\geq 0$ for $\gamma=0$, and that definiteness is lost for sufficiently large $\gamma$. However, the interaction of the BCs with this behavior makes the problem interesting and non-trivial to solve. We will compute upper and lower bounds for the optimal $\gamma$ in~\eqref{E:ProbEx2} in \S\ref{S:ShearFlowEx}.
%

\section{Outer SDP relaxations}
\label{S:OuterApprox}

Our first approach to solve~\eqref{E:OptProblem} is to derive a sequence of \emph{outer} approximations for its feasible set, defined as
\begin{equation}
\label{E:FeasSetDef}
T \defeq \left\{ \vec{\gamma} \in \mathbb{R}^s :\,\,\forall \vec{w}\in H,\,\, \mathcal{F}_{\vec{\gamma}}\{\vec{w}\}  \geq 0 \right\}.
\end{equation}
In other words, we look for a family of sets $\{T^\mathrm{out}_N\}_{N\geq 0}$ such that $T \subset T^\mathrm{out}_N$. Optimizing the cost function over $T^\mathrm{out}_N$ then gives a lower bound for the optimal value of~\eqref{E:OptProblem}. 

The outer approximation set $T^\mathrm{out}_N$ can be found by 
considering a polynomial truncation of $\vec{w}\in H$ of degree $N$. In particular, suppose that
\begin{equation}
\label{E:SNdef}
\vec{w} = 
[u,v]^T
\in S_N \defeq H \cap \left( \mathcal{P}_N \times \mathcal{P}_N \right) \subset H,
\end{equation}
where $\mathcal{P}_N$ is the set of polynomials of degree less than or equal to $N$ on $[-1,1]$.
Note that $S_N$ is non-empty for any degree bound $N$ because $H$ contains the zero polynomial, and it contains nonzero elements if $N$ is large enough to guarantee  sufficient degrees of freedom to satisfy the BCs prescribed on $H$ in~\eqref{E:HspaceDef}. Finally, $S_N \subset S_{N+1}$ because $\mathcal{P}_N \subset \mathcal{P}_{N+1}$.

Now, let $\hat{u}_0,\,\hdots,\,\hat{u}_N$ and $\hat{v}_0,\,\hdots,\,\hat{v}_N$ be the coefficients representing the polynomials $u$ and $v$ in any chosen basis for $\mathcal{P}_N$, and define $\vec{\varphi}_N\defeq \left[
\hat{u}_0,\, \hdots,\, \hat{u}_N,\,\hat{v}_0,\, \hdots,\, \hat{v}_N
\right]^T$. Since $\mathcal{F}_{\vec{\gamma}}$ in~\eqref{E:OptProblem} is quadratic and the constraints imposed on $H$ are linear, it is clear that there exist a matrix $\vec{Q}_N(\vec{\gamma})$, affine in $\vec{\gamma}$, such that
\begin{equation*}
\mathcal{F}_{\vec{\gamma}}\{\vec{w}\} = {\vec{\varphi}_N}^T \vec{Q}_N(\vec{\gamma})\vec{\varphi}_N,
\end{equation*}
and a matrix $\vec{A}_N$ such that
\begin{equation*}
\vec{w} \in S_N \Leftrightarrow \vec{A}_N \vec{\varphi}_N = \vec{0}.
\end{equation*}
Upon selecting a matrix $\vec{\Pi}_N$ satisfying $\mathcal{R}(\vec{\Pi}_N)=\mathcal{N}(\vec{A}_N)$, it follows that
\begin{equation}
T^\mathrm{out}_N 
\defeq \left\{ \vec{\gamma} \in \mathbb{R}^s :\,\,\forall \vec{w}\in S_N,\,\, \mathcal{F}_{\vec{\gamma}}\{\vec{w}\}  \geq 0 \right\}
=\left\{ \vec{\gamma} \in \mathbb{R}^s :\,\, {\vec{\Pi}_N}^T  \vec{Q}_N(\vec{\gamma})\vec{\Pi}_N \succeq 0\right\},
\label{E:OuterApproxSet}
\end{equation}
and since $S_N \subset S_{N+1} \subset H$, the feasible set $T$ of~\eqref{E:OptProblem}, defined as in~\eqref{E:FeasSetDef}, satisfies
\begin{equation*}
T \subset T^\mathrm{out}_{N+1} \subset T^\mathrm{out}_N, \quad N\in\mathbb N.
\end{equation*}

This suggests that a sequence of lower bounds on the optimal value of~\eqref{E:OptProblem} can be found by solving a series of truncated optimization problems.
\begin{theorem}
\label{T:ConvergenceOuterSDPrelax}
Let $p^*$ be the optimal value of~\eqref{E:OptProblem} and, for each integer $N$, let $p_N^*$ be the optimal value of the SDP
\begin{equation}
\label{E:OuterSDP}
\begin{aligned}
\min_{\vec{\gamma}} \quad & \vec{c}^T\vec{\gamma}	\\
\text{s.t.} \quad & \vec{\Pi}_N^T\,\vec{Q}_{N}(\vec{\gamma})\,\vec{\Pi}_N \succeq 0.
\end{aligned}
\end{equation}
Then, $\{p^\star_N\}_{N\geq 0}$ is a non-decreasing sequence of lower bounds for $p^\star$. Furthermore, if a minimizer $\gamma^\star$ exists in~\eqref{E:OptProblem}, then  $\displaystyle\lim_{N\to \infty} \vert p^\star_N - p^\star \vert = 0 $.
\end{theorem}

\begin{proof}
See Appendix~\ref{A:ProofConvergenceOuterAprox}.
\end{proof}

\begin{remark}
Clearly, infeasibility of the SDP~\eqref{E:OuterSDP} for a certain $N$ provides a \emph{certificate of infeasibility} for~\eqref{E:OptProblem}. However, note that the feasibility (resp. unboundedness) of~\eqref{E:OuterSDP}  for any finite $N$ does \emph{not}  prove that~\eqref{E:OptProblem} is feasible (resp. unbounded).
\end{remark}

It is important to note that Theorem~\ref{T:ConvergenceOuterSDPrelax}
provides no control on the gap $p^\star - p^\star_N$ \emph{as a function of} $N$. In other words, an arbitrarily large $N$ might be required for a given level of approximation accuracy. Consequently, the rest of this work will focus on proving checkable conditions upon which upper bounds can be placed on $p^*$.

\section{Inner SDP relaxations}
\label{S:InnerApprox}

Upper bounds on the optimal value of~\eqref{E:OptProblem} that complement the lower bounds from~Theorem~\ref{T:ConvergenceOuterSDPrelax} can be found by optimizing the cost function over an inner approximation $T^\mathrm{in}_N$ of the true feasible set. Such an inner approximation can be constructed by replacing the integral inequality $\mathcal{F}_{\vec{\gamma}}\{\vec{w}\} \geq 0$ with a stronger, but tractable, integral inequality over the space $H$ in~\eqref{E:HspaceDef}. This strategy is complementary to the approach followed in \S\ref{S:OuterApprox}, where we effectively replaced the space $H$ with a tractable subspace $S_N$.
In particular, we look for a  lower bound $\mathcal{F}_{\vec{\gamma}}\{\vec{w}\} \geq \mathcal{G}_{\vec{\gamma}}\{\vec{w}\}$, where $\mathcal{G}_{\vec{\gamma}}\{\vec{w}\}$ is a functional whose non-negativity over $H$ can be enforced via a set of LMIs. Any $\vec{\gamma}$ such that $\mathcal{G}_{\vec{\gamma}}\{\vec{w}\} \geq 0$ on $H$ is then also feasible  for~\eqref{E:OptProblem}, and the corresponding cost $\vec{c}^T\vec{\gamma}$ is an upper bound for the optimal value of~\eqref{E:OptProblem}.

\subsection{Legendre series expansions}
\label{S:LegendreExpansions}

The key to constructing an inner approximation for the problem~\eqref{E:OptProblem} is to construct a functional  $\mathcal{G}_{\vec{\gamma}}:H\to\mathbb{R}$ such that $\mathcal{F}_{\vec{\gamma}}\{\vec{w}\} \geq \mathcal{G}_{\vec{\gamma}}\{\vec{w}\}$ for all $\vec{w}\in H$. To do this, we expand the components $u$ and $v$ of $\vec{w}$ (recall our simplifying restriction to the two-dimensional case) in terms of Legendre polynomials. That is, we write expansions such as
\begin{equation}
\label{E:FullLegExp}
\partial^\alpha u = \sum_{n=0}^{\infty} \hat{u}^\alpha_n \mathcal{L}_n(x),
\end{equation}
where $\mathcal{L}_n(x)$ is the Legendre polynomial of degree $n$ and $\hat{u}^\alpha_n$ is the $n$-th Legendre coefficient. Similar expressions can be written for $v$ and its derivatives. 

Legendre series expansions are useful because the Legendre polynomials are orthogonal on $[-1,1]$, i.e., $\int_{-1}^{1}\mathcal{L}_m\,\mathcal{L}_n \,{\rm d}x = 0$ if $m\neq n$~\cite{Jackson1930}. This will enable us to enforce the non-negativity of the functional $\mathcal{F}_{\vec{\gamma}}$ in~\eqref{E:OptProblem} with a set of finite-dimensional, numerically tractable conditions. Note that although other polynomial basis functions, e.g. Chebyshev polynomials, may have more attractive numerical properties and may be more appropriate to implement the outer SDP relaxation of Theorem~\ref{T:ConvergenceOuterSDPrelax}, they cannot be used here  because they are only orthogonal with respect to a weighting function.
A short introduction to Legendre polynomials, Legendre series and their properties is given in Appendix~\ref{A:LegPolys}; see~\cite{Jackson1930,Zeidler1995,Agarwal2009} for a comprehensive treatment of the subject.

To avoid working with infinite series and to facilitate our analysis, we decompose~\eqref{E:FullLegExp} into a finite sum and a remainder function. More precisely, given an integer $i$ we define the remainder function
\begin{equation}
\label{E:RemFunDef}
U^\alpha_i(x) = \sum_{n=i+1}^{\infty} \hat{u}^\alpha_n \mathcal{L}_n(x).
\end{equation}
Next, we choose an integer $N$ such that
\begin{equation}
\label{E:Ncondition}
N \geq d_F+ k -1,
\end{equation}
where $d_F$ is the degree of the polynomial matrix $\vec{F}$ defined in~\eqref{E:FmatDef}. For each $\alpha\in\{1,\,\hdots,\,k\}$ we decompose the Legendre expansion of $\partial^\alpha u$  as
\begin{align}
\label{E:LegExpansion}
\der{\alpha}u &= \sum_{n=0}^{N+\alpha} \hat{u}^\alpha_n \, 
\Leg_n(x) + U^\alpha_{N+\alpha}(x).
\end{align}
For notational ease, we record the Legendre coefficients $\hat{u}_r^\alpha,\,\hdots,\,\hat{u}_s^\alpha$ for any two integers $0\leq r\leq s$ in the vector
\begin{equation}
\label{E:LegCoeffVector}
\vec{\hat{u}}^\alpha_{[r,s]} = 
\begin{bmatrix} 
\hat{u}_r^\alpha, & \ldots, & \hat{u}_{s}^\alpha
 \end{bmatrix}^T \in \mathbb{R}^{s-r+1}.
\end{equation}

For technical reasons that will be pointed out in \S\ref{S:IntTermExpansion}, it will also be convenient to introduce an ``extended'' decomposition for the highest-order derivative, $\der{k}u$. Specifically, let
\begin{equation}
\label{E:Mdef}
M  \defeq N+2k+d_F
\end{equation}
and consider
\begin{equation}
\label{E:DerKExpansion}
\der{k}u = \sum_{n=0}^{M} \hat{u}^k_n \, \Leg_n(x) + 
U^k_M(x). 
\end{equation} 
 
The following result, proven in~Appendix~\ref{A:ProofIntegrationLemma}, relates the Legendre coefficients of $u,\,\partial u,\,\hdots,\,\partial^k u$.
\begin{lemma}
\label{T:IntegrationLemma}
Let $u\in C^m([-1,1])$ and its derivatives up to order $k\leq m-1$ be expanded as in~\eqref{E:LegExpansion}, and let $M$ be as in~\eqref{E:Mdef}. For any $\alpha\in\{1,\,\hdots,\,k\}$ and any two integers $r,\,s$ with $0\leq r\leq s \leq M+\alpha-k$, there exist matrices $\vec{B}^\alpha_{[r,s]}$ and $\vec{D}^\alpha_{[r,s]}$ such that 
\begin{equation*}
\vec{\hat{u}}^\alpha_{[r,s]} = \vec{B}^\alpha_{[r,s]} \,\mathcal{D}^{k-1}u\at{-1} + \vec{D}^\alpha_{[r,s]} \vec{\hat{u}}^k_{[0,M]}.
\end{equation*}
Furthermore, $\vec{B}^\alpha_{[r,s]}=\vec{0}$ if $r\geq k-\alpha$.
\end{lemma}

This lemma simply states that given the Legendre coefficients $\hat{u}^k_0,\,\hdots,\,\hat{u}^k_M$ of $\partial^ku$, the Legendre coefficients of all derivatives of order $\alpha<k$ can be computed uniquely if the boundary values $\mathcal{D}^{k-1}u\at{-1} $ are specified. These boundary values play the role of integration constants, and should be treated as variables until specific BCs are prescribed. Given an integer $n$, we therefore define the vector of variables
\begin{equation}
\label{E:uCheckDef}
\vec{\check{u}}_{n} = \begin{bmatrix}
\left(\mathcal{D}^{k-1} u\at{-1}\right)^T, & \hat{u}^k_0, & \ldots, & \hat{u}^k_n
\end{bmatrix}^T
\in \mathbb{R}^{k+n+1}.
\end{equation}

The boundary values of $u$ and its derivatives can also be represented in terms of our Legendre expansions. This is useful because the integral inequality in~\eqref{E:OptProblem} is only required to hold for functions that satisfy prescribed BCs. The following result is proven in Appendix~\ref{A:ProofBCLemma}.
\begin{lemma}
\label{T:BCLemma}
Let $u\in C^m([-1,1])$ and its derivatives up to order $k\leq m-1$ be expanded as in~\eqref{E:LegExpansion}, and let $\mathcal{B}^{k-1} u \in \mathbb{R}^{2k}$ be defined according to~\eqref{E:BoundaryVector}. Moreover, let $M$ be as in~\eqref{E:Mdef}, and let $\vec{\check{u}}_{M}\in\mathbb{R}^{k+M+1}$ be defined according to~\eqref{E:uCheckDef}. There exists a matrix $\vec{G}_M \in \mathbb{R}^{2k\times (k+M+1)}$ such that 
$\mathcal{B}^{k-1} u = \vec{G}_M \vec{\check{u}}_{M}$.
\end{lemma}

\subsection{Legendre expansions of $\mathcal{F}_{\vec{\gamma}}\{\vec{w}\}$}
\label{S:IntTermExpansion}

Recalling the definition of $\mathcal{D}^{\vec{k}}\vec{w}$, we see from~\eqref{E:OptProblem} that $\mathcal{F}_{\vec{\gamma}}\{\vec{w}\}$ is a sum of elementary terms of the form
\begin{equation}
\label{E:GenericTerm}
\int_{-1}^{1} f\,\partial^\alpha u \, \partial^\beta v \,{\rm d}x,
\end{equation}
where $\alpha,\beta\in\{0,\,\hdots,\,k\}$. Here, $f=f(x;\vec{\gamma})$ denotes the appropriate entry of the integrand matrix $\vec{F}(x;\vec{\gamma})$ and, consequently, it is a polynomial of degree at most $d_F$ whose coefficients are affine in $\vec{\gamma}$. We consider a term involving both components $u$ and $v$ of $\vec{w}$ for generality, but the following arguments also hold when $\partial^\alpha u\, \partial^\beta v$ is replaced with $\partial^\alpha u\, \partial^\beta u$ or $\partial^\alpha v\, \partial^\beta v$. 

For each term of the form~\eqref{E:GenericTerm}, we substitute $\partial^\alpha u$ and $\partial^\beta v$ with their decomposed Legendre expansions according to the following strategy:
\begin{itemize}
\item{If $\alpha\neq k$ \emph{or} $\beta\neq k$, use~\eqref{E:LegExpansion}.}
\item{If $\alpha=\beta=k$, use the ``extended'' decomposition~\eqref{E:DerKExpansion}.}
\end{itemize}
The reasons for this choice will be explained in Remark~\ref{R:ExpDer}, after Lemma~\ref{T:LemmaQ}. In either case, we can rewrite~\eqref{E:GenericTerm} as
\begin{equation}
\label{E:GenericFiTermExpansion}
\int_{-1}^{1}  f\, \partial^\alpha u\, \partial^\beta v \,{\rm d}x = 
\mathcal{P}^{\alpha\beta}_{uv} + 
\mathcal{Q}^{\alpha\beta}_{uv} + \mathcal{R}^{\alpha\beta}_{uv},
\end{equation}
where
\begin{subequations}
\begin{align}
\mathcal{P}^{\alpha\beta}_{uv} &= \sum_{m=0}^{N_\alpha}\sum_{n=0}^{N_\beta} 
\hat{u}^\alpha_m \hat{v}^\beta_n 
\int_{-1}^{1} f\, \mathcal{L}_m \, \mathcal{L}_n \,{\rm d}x,
\label{E:Pdef}
\\
\mathcal{Q}^{\alpha\beta}_{uv} &=  
\sum_{n=0}^{N_\alpha}  \hat{u}^\alpha_n 
\int_{-1}^{1} f\, \mathcal{L}_n \, V^\beta_{N_\beta} \,{\rm d}x
+\sum_{n=0}^{N_\beta}  \hat{v}^\beta_n 
\int_{-1}^{1} f \, \mathcal{L}_n \, U^\alpha_{N_\alpha}\,{\rm d}x,
\label{E:Qdef}
\\
\mathcal{R}^{\alpha\beta}_{uv} &= 
\int_{-1}^{1} f\, U^\alpha_{N_\alpha}\, V^\beta_{N_\beta}\,{\rm d}x.
\label{E:Rdef}
\end{align}
\end{subequations}
Here and in the following it should be understood that $N_\alpha=N+\alpha$ and $N_\beta=N+\beta$ 
if~\eqref{E:LegExpansion} is used to expand $\partial^\alpha u$ and $\partial^\beta v$, while $N_\alpha = N_\beta = M=N+2k+d_F$ if~\eqref{E:DerKExpansion} is used. 

The term $\mathcal{P}^{\alpha\beta}_{uv}$ is finite dimensional, and for any choice of $\alpha,\beta\in\{0,\hdots,k\}$ it can be rewritten as a symmetric quadratic form for the vectors $\vec{\hat{u}}^\alpha_{[0,N_\alpha]}$ and $\vec{\hat{v}}^\beta_{[0,N_\beta]}$.  Recalling Lemma~\ref{T:IntegrationLemma} and defining
\begin{equation}
\label{E:PsiDef}
\vec{\psi}_{M} \defeq 
\begin{bmatrix} \vec{\check{u}}_{M} \\  \vec{\check{v}}_{M}\end{bmatrix} 
\in \mathbb{R}^{2(k+M+1)},
\end{equation}
where $\vec{\check{u}}_{M}$ and $\vec{\check{v}}_{M}$ are as in~\eqref{E:uCheckDef}, we arrive at the following result.
\begin{lemma}
\label{T:LemmaP}
Let $\mathcal{P}^{\alpha\beta}_{uv}$ be as in~\eqref{E:Pdef} and $\vec{\psi}_M$ be defined according to~\eqref{E:PsiDef}. There exists a matrix 
$\vec{P}^{\alpha\beta}_{uv}(\vec{\gamma}) \in \mathbb{S}^{2(k+M+1)}$, whose entries are affine in $\vec{\gamma}$, such that
\begin{equation*}
\mathcal{P}^{\alpha\beta}_{uv} = 
{\vec{\psi}_M}^T \, \vec{P}^{\alpha\beta}_{uv}(\vec{\gamma})\,\vec{\psi}_M.
\end{equation*}
\end{lemma}

The term $\mathcal{Q}^{\alpha\beta}_{uv}$ is less straightforward to handle, because it couples the first $N_\alpha+1$ and $N_\beta+1$ modes of $\partial^\alpha u$ and $\partial^\beta v$, respectively, to the remainder functions  $V^\beta_{N_\beta}$ and $U^\alpha_{N_\alpha}$. 
We show in Appendix~\ref{A:ProofLemmaQ} that considering the extended decomposition~\eqref{E:DerKExpansion} for the Legendre series of $\partial^k u$ and $\partial^k v$ enables us to write $\mathcal{Q}^{\alpha\beta}_{uv}$ as a finite-dimensional matrix quadratic form for the vector $\vec{\psi}_M$ if $\alpha\neq k$ or $\beta\neq k$. If $\alpha=\beta=k$, on the other hand, we cannot do the same unless $f$ in~\eqref{E:Qdef} is independent of $x$ (in this case, the orthogonality of the Legendre polynomials and the remainder functions implies that $\mathcal{Q}^{kk}_{uv}=0$). Instead, we estimate $\mathcal{Q}^{kk}_{uv}$ to decouple the remainder functions from the other terms. 

To make these ideas more precise, let us introduce a family of ``deflation'' matrices $\vec{L}_n$ such that
\begin{equation}
\label{E:DeflMatDef}
\vec{L}_n \,\vec{\psi}_M = \begin{bmatrix} \vec{\hat{u}}^k_{[n,M]} \\ \vec{\hat{v}}^k_{[n,M]} \end{bmatrix}, \quad n\in\{0,\,\hdots,\,M\},
\end{equation}
and $\vec{L}_n \,\vec{\psi}_M = \vec{0}$ if $n > M$. The existence of $\vec{L}_n$ follows from~\eqref{E:PsiDef}, \eqref{E:uCheckDef}, and~\eqref{E:LegCoeffVector}. 
Moreover, given four integers $a\leq b$ and $c\leq d$, let 
$\vec{\Phi}{}_{[a,b]}^{[c,d]}$ be 
a $(b-a+1)\times(d-c+1)$ matrix whose $ij$-th element is defined as
\begin{equation}
\label{E:PhiMatDef}
\left(\vec{\Phi}{}_{[a,b]}^{[c,d]}\right)_{ij} = \int_{-1}^{1} f\,\mathcal{L}_{m_i}\, 
\mathcal{L}_{n_j} \,{\rm d}x,
\end{equation}
where $m_i$ and $n_j$ are the $i$-th and $j$-th elements of the sequences 
$\{a,\,\hdots,\,b\}$ and $\{c,\,\hdots,\,d\}$. Note that, strictly speaking, $\vec{\Phi}{}_{[c,d]}^{[a,b]}$ depends on $f$, and its entries are affine on $\vec{\gamma}$. We do not indicate such dependencies explicitly to avoid complicating our notation further. 
The following result is proven in Appendix~\ref{A:ProofLemmaQ}.

\begin{lemma}
\label{T:LemmaQ}
Let $\mathcal{Q}^{\alpha\beta}_{uv}$ be as in~\eqref{E:Qdef} and let $d_F$ be the degree of $f(x;\vec{\gamma})$.
\begin{enumerate}[(i)]
\item If $\alpha\neq k$ or $\beta\neq k$, there exists a matrix $
\vec{Q}^{\alpha\beta}_{uv}(\vec{\gamma}) \in \mathbb{S}^{2(k+M+1)}$, whose entries are affine in $\vec{\gamma}$, such that
\begin{equation*}
\mathcal{Q}^{\alpha\beta}_{uv} = 
{\vec{\psi}_M}^T \, \vec{Q}^{\alpha\beta}_{uv}(\vec{\gamma})\,\vec{\psi}_M.
\end{equation*}
\item If $\alpha=\beta=k$, let $\overline{M}\defeq M+1-d_F$, define $\vec{\Delta}\in \mathbb{S}^{d_F}$ as
\begin{equation*}
\vec{\Delta}\defeq
\mathrm{Diag}\left(\frac{2}{2(M+1)+1},\,\hdots,\,\frac{2}{2(M+d_F)+1}\right) , 
\end{equation*}
and define $\vec{Y}(\vec{\gamma}) \in \mathbb{R}^{2d_F \times 2d_F}$ as 
\begin{equation*}
\vec{Y}(\vec{\gamma}) \defeq
\frac{1}{2}\begin{bmatrix}
\vec{0} & \vec{\Phi}{}_{[M+1-d_F,M]}^{[M+1,M+d_f]} \\ \vec{\Phi}{}_{[M+1-d_F,M]}^{[M+1,M+d_f]} & \vec{0}
\end{bmatrix}.
\end{equation*}
Finally, let $\vec{Q}^{kk}_{uv}\in\mathbb{S}^{2d_F}$ and a diagonal matrix 
$\vec{\Sigma}^{kk}_{uv} \in \mathbb{S}^2$
satisfy the LMI
\begin{equation}
\label{E:AuxiliaryLMI}
\vec{\Omega}(\vec{Q}^{kk}_{uv},\vec{\Sigma}^{kk}_{uv},\vec{\gamma})\defeq 
\begin{bmatrix}
\vec{Q}^{kk}_{uv} & \vec{Y}(\vec{\gamma}) \\
\vec{Y}(\vec{\gamma})^T & \vec{\Sigma}^{kk}_{uv}\otimes\vec{\Delta}
\end{bmatrix}
\succeq 0,
\end{equation}
where $\otimes$ is the usual Kronecker product. Then,
$\mathcal{Q}^{kk}_{uv}$ can be bounded as
\begin{equation}
\label{E:LemmaQEstimate}
\mathcal{Q}^{kk}_{uv} \geq 
-{\vec{\psi}_M}^T 
\left( {\vec{L}_{\overline{M}}}^T\, \vec{Q}^{kk}_{uv} \,\vec{L}_{\overline{M}} \right)
{\vec{\psi}_M}
-\int_{-1}^{1}
\begin{bmatrix}U^k_M\\V^k_M\end{bmatrix}^T 
\vec{\Sigma}^{kk}_{uv} 
\begin{bmatrix}U^k_M\\V^k_M\end{bmatrix} \dx.
\end{equation}
\end{enumerate}
\end{lemma}

\begin{remark}
The LMI~\eqref{E:AuxiliaryLMI} was chosen such that~\eqref{E:LemmaQEstimate}, essentially its Schur complement condition, separates the contributions of $\vec{\psi}_M$, $U_M^k$ and $V_M^k$. 
As will be demonstrated in \S\ref{S:FeasSEtApproxEx}, inequality~\eqref{E:LemmaQEstimate} is the main source of conservativeness. To make~\eqref{E:LemmaQEstimate} as sharp as possible, we consider $\vec{Q}^{kk}_{uv}$ and $\vec{\Sigma}^{kk}_{uv}$ as auxiliary variables, to be determined subject to~\eqref{E:AuxiliaryLMI}.
\end{remark}

\begin{remark} 
\label{R:ExpDer}
$\mathcal{Q}_{uv}^{\alpha\beta}$ can be represented exactly only if we consider all Legendre coefficients of $\partial^ku$, $\partial^k v$ up to order $M$ explicitly: this is what motivates the use of the extended decomposition~\eqref{E:DerKExpansion} for these functions. Moreover, note that instead of using the bound~\eqref{E:LemmaQEstimate} we could write $\mathcal{Q}_{uv}^{kk}$ exactly in terms of $\vec{\psi}_{M+d_F}$, but this does not suit our aims because $\vec{\psi}_{M+d_F}$ is not decoupled from $U_M^k$, $V_M^k$ (the Legendre coefficients $\hat{u}^k_{M+i}$, $1\leq i \leq d_F$ appear in the definition of $U_M^k$).
\end{remark}

Lemmas~\ref{T:LemmaP} and~\ref{T:LemmaQ} show that $\mathcal{P}_{uv}^{\alpha\beta}$ and $\mathcal{Q}_{uv}^{\alpha\beta}$ can be expressed or bounded using $\vec{\psi}_M$, $U^k_{M}$ and $V^k_{M}$ for any $\alpha,\beta\in\{0,\hdots,k\}$. 
If $\alpha=\beta=k$,~\eqref{E:Rdef} also depends $U^k_{M}$ and $V^k_{M}$. The following result, proven in Appendix~\ref{A:ProofLemmaR}, shows that $\mathcal{R}_{uv}^{\alpha\beta}$ can be bounded using the same quantities when $\alpha\neq k$ or $\beta \neq k$.

\begin{lemma}
\label{T:LemmaR}
Suppose $\alpha\neq k$ or $\beta \neq k$, and let $\vec{\hat{f}}(\vec{\gamma})=[\hat{f}_1(\vec{\gamma}),,\cdots,\, \hat{f}_{d_F}(\vec{\gamma})]^T$ be the vector of Legendre coefficients of the polynomial $f$. 
There exist a positive semidefinite matrix 
$\vec{R}^{\alpha\beta}_{uv}\in\mathbb{S}^{2(M+k+1)}$ with 
$\|\vec{R}^{\alpha\beta}_{uv}\|_F \sim N^{\alpha+\beta-2k-1}$
and a positive definite matrix $\vec{\Sigma}^{\alpha\beta}_{uv}\in\mathbb{S}^{2}$ with
$\|\vec{\Sigma}^{\alpha\beta}_{uv}\|_F \sim N^{\alpha+\beta - 2k}$ such that 
$\mathcal{R}^{\alpha\beta}_{uv}$ is bounded as 
\begin{equation}
\label{E:Restimate}
\left\vert \mathcal{R}^{\alpha\beta}_{uv} \right\vert\leq 
\|\vec{\hat{f}}(\vec{\gamma})\|_1\,
{\vec{\psi}_M}^T\, 
\vec{R}^{\alpha\beta}_{uv} \,
{\vec{\psi}_M}
+
\|\vec{\hat{f}}(\vec{\gamma})\|_1 
\int_{-1}^{1}
\begin{bmatrix}U^k_M\\V^k_M\end{bmatrix}^T 
\vec{\Sigma}^{\alpha\beta}_{uv} 
\begin{bmatrix}U^k_M\\V^k_M\end{bmatrix} \dx.
\end{equation}
\end{lemma}

\begin{remark}
The scaling of the Frobenius norms of $\vec{R}^{\alpha\beta}_{uv}$ and $\vec{\Sigma}^{\alpha\beta}_{uv}$ with $N$ reflects the fact that the magnitude of $\mathcal{R}_{uv}^{\alpha\beta}$ diminishes to zero as $N$ is raised. In contrast to Lemma~\ref{T:LemmaQ}, consequently, the conservativeness of the estimates in Lemma~\ref{T:LemmaR} can be reduced by simply increasing $N$.
\end{remark}

\subsection{A lower bound for $\mathcal{F}_{\vec{\gamma}}\{\vec{w}\}$}
\label{SS:LowerBoundF}
Let us now combine Lemmas~\ref{T:LemmaP}--\ref{T:LemmaR} to find 
a lower bounding functional $\mathcal{G}_{\vec{\gamma}}$ for the integral functional $\mathcal{F}_{\vec{\gamma}}$ in~\eqref{E:OptProblem}. 
To account for the different cases in Lemma~\ref{T:LemmaQ}, we consider the contributions from terms with $\alpha=\beta=k$ first. 

Let $\vec{S}(x;\vec{\gamma})$ be the symmetric matrix obtained from the rows and columns of the matrix $\vec{F}(x;\vec{\gamma})$ in~\eqref{E:OptProblem} corresponding to the entries $\partial^k u$ and $\partial^k v$ of $\mathcal{D}^k\vec{w}$. The contribution of the terms with $\alpha=\beta=k$ to $\mathcal{F}_{\vec{\gamma}}\{\vec{w}\}$ is
\begin{equation*}
\int_{-1}^{1}  
\begin{bmatrix} \partial^k u \\ \partial^k v\end{bmatrix}^T
\vec{S}(x;\vec{\gamma})
\begin{bmatrix}\partial^k u \\ \partial^k v\end{bmatrix}
\,{\rm d}x. 
\end{equation*}
It follows from Lemma~\ref{T:LemmaP} and part (ii) of Lemma~\ref{T:LemmaQ} that
\begin{multline}
\int_{-1}^{1}  
\begin{bmatrix} \partial^k u \\ \partial^k v\end{bmatrix}^T
\vec{S}
\begin{bmatrix}\partial^k u \\ \partial^k v\end{bmatrix}
\,{\rm d}x \geq
{\vec{\psi}_M}^T 
\left[
\vec{P}_{uu}^{kk} + 2\vec{P}_{uv}^{kk} + \vec{P}_{vv}^{kk}
- {\vec{L}_{\overline{M}}}^T
\left( \vec{Q}^{kk}_{uu}  +2\vec{Q}^{kk}_{uv} + \vec{Q}^{kk}_{vv}  
\right) \vec{L}_{\overline{M}} 
\right]\vec{\psi}_M
\\
+ \int_{-1}^{1}
\begin{bmatrix}U^k_{M}\\V^k_{M}\end{bmatrix}^T 
\left[ \vec{S}
-\vec{\Sigma}^{kk}_{uu}  -2\vec{\Sigma}^{kk}_{uv} - \vec{\Sigma}^{kk}_{vv}\right]
\begin{bmatrix}U^k_{M}\\V^k_{M}\end{bmatrix} \dx,
\label{E:LowerBound1}
\end{multline}
where the auxiliary variables $\vec{Q}^{kk}_{uu}$, $\vec{\Sigma}^{kk}_{uu}$, $\vec{Q}^{kk}_{uv}$, $\vec{\Sigma}^{kk}_{uv}$, $\vec{Q}^{kk}_{vv}$, and $\vec{\Sigma}^{kk}_{vv}$ must satisfy three LMIs defined as in~\eqref{E:AuxiliaryLMI}. For notational convenience, we let 
\begin{equation}
\label{E:AuxVarDef}
\mathcal{Y} = \left\{\vec{Q}^{kk}_{uu},\vec{\Sigma}^{kk}_{uu},  \vec{Q}^{kk}_{uv},\vec{\Sigma}^{kk}_{uv},\vec{Q}^{kk}_{vv},\vec{\Sigma}^{kk}_{vv} \right\}
\end{equation}
be the list of all auxiliary variables, and we combine the three LMIs they must satisfy into the equivalent block-diagonal LMI
\begin{equation}
\label{E:CompoundLMI}
\overline{\vec{\Omega}}(\vec{\gamma},\mathcal{Y})\defeq 
\begin{bmatrix}
\vec{\Omega}(\vec{Q}^{kk}_{uu},\vec{\Sigma}^{kk}_{uu},\vec{\gamma}) & \vec{0}& \vec{0}\\
\vec{0}& \vec{\Omega}(\vec{Q}^{kk}_{uv},\vec{\Sigma}^{kk}_{uv},\vec{\gamma}) & \vec{0}\\
\vec{0}& \vec{0} & \vec{\Omega}(\vec{Q}^{kk}_{vv},\vec{\Sigma}^{kk}_{vv},\vec{\gamma})
\end{bmatrix} 
\succeq 0.
\end{equation}

All terms contributing to $\mathcal{F}_{\vec{\gamma}}\{\vec{w}\}$ with $\alpha\neq k$ or $\beta\neq k$ can instead be lower bounded using Lemmas~\ref{T:LemmaP}--\ref{T:LemmaR} to obtain expressions 
such as
\begin{multline}
\label{E:LowerBound2}
\int_{-1}^{1}  f\, \partial^\alpha u\, \partial^\beta v \,{\rm d}x 
\geq {\vec{\psi}_M}^T \,
\left( \vec{P}_{uv}^{\alpha\beta} + \vec{Q}_{uv}^{\alpha\beta}
-\|\vec{\hat{f}}(\vec{\gamma})\|_1 \vec{R}^{\alpha\beta}_{uv} 
\right)
\,\vec{\psi}_M
\\
-\|\vec{\hat{f}}(\vec{\gamma})\|_1 
\int_{-1}^{1}
\begin{bmatrix}U^k_{M}\\V^k_{M}\end{bmatrix}^T 
\vec{\Sigma}^{\alpha\beta}_{uv} 
\begin{bmatrix}U^k_{M}\\V^k_{M}\end{bmatrix} \dx.
\end{multline}

From~\eqref{E:LowerBound1} and~\eqref{E:LowerBound2}  we conclude that it is possible to construct a matrix $\vec{Q}_M = \vec{Q}_M(\vec{\gamma},\mathcal{Y})\in\mathbb{S}^{2(k+M+1)}$ and a positive definite matrix $\vec{\Sigma}_M = \vec{\Sigma}_M(\vec{\gamma},\mathcal{Y})\in\mathbb{S}^2$, such that for all $\vec{w}$
\begin{equation}
\label{E:LowerBoundIntegralTerm}
\mathcal{F}_{\vec{\gamma}}\{\vec{w}\} \geq 
{\vec{\psi}_M}^T \vec{Q}_M \vec{\psi}_M 
+\int_{-1}^{1}\begin{bmatrix} U^k_{M} \\ V^k_{M} \end{bmatrix}^T
(\vec{S}- \vec{\Sigma}_M) 
\begin{bmatrix} U^k_{M} \\ V^k_{M} \end{bmatrix}
\dx.
\end{equation}
Note that $\vec{Q}_{M}$ and $\vec{\Sigma}_M$ are affine with respect to the variables listed in $\mathcal{Y}$ but \emph{not} in $\vec{\gamma}$, because Lemma~\ref{T:LemmaR} introduces absolute values of linear functions of $\vec{\gamma}$.

\subsection{Projection onto the boundary conditions}
\label{SS:BCproj} 

The lower bound~\eqref{E:LowerBoundIntegralTerm} holds for any continuously differentiable function $\vec{w}$, irrespectively of whether it satisfies the BCs prescribed on $H$. 
Recalling~\eqref{E:HspaceDef}, these are given by the set of $p$ homogeneous equations 
\begin{equation}
\label{E:BCinit}
\vec{A} \mathcal{B}^{\vec{l}}\vec{w}=\vec{0}.
\end{equation} 

To enforce as many BCs as possible in~\eqref{E:LowerBoundIntegralTerm} and to sharpen the lower bound over the space $H$, we need to rewrite~\eqref{E:BCinit} in terms of our Legendre expansions. 
We begin by introducing a permutation matrix $\vec{P}$ such that
\begin{equation}
\label{E:PermMatDef}
\mathcal{B}^{\vec{l}}\vec{w} = \vec{P} 
\begin{bmatrix} 
\mathcal{B}^{\vec{k-1}}\vec{w} \\ \mathcal{B}^{[\vec{k},\vec{l}]}\vec{w} 
\end{bmatrix}.
\end{equation}
so~\eqref{E:BCinit} becomes
\begin{equation}
\label{E:BCprojEq1}
\vec{A} \vec{P} 
\begin{bmatrix} 
\mathcal{B}^{\vec{k-1}}\vec{w} \\ \mathcal{B}^{[\vec{k},\vec{l}]}\vec{w} 
\end{bmatrix} =  \vec{0}.
\end{equation}
A straightforward corollary of Lemma~\ref{T:BCLemma} and~\eqref{E:PsiDef} is that there exists a matrix $\vec{J}$ such that $\mathcal{B}^{\vec{k-1}}\vec{w} = \vec{J}\vec{\psi}_M$. Then,~\eqref{E:BCprojEq1} can be rewritten as
\begin{equation}
\label{E:newBCs}
\vec{K}
\begin{bmatrix} 
\vec{\psi}_M \\ \mathcal{B}^{[\vec{k},\vec{l}]}\vec{w} 
\end{bmatrix} =  \vec{0}, \quad
\vec{K} \defeq \vec{A} \vec{P} 
\begin{bmatrix} 
\vec{J} & \vec{0} \\ \vec{0} & \vec{I}
\end{bmatrix}.
\end{equation}

From~\eqref{E:newBCs} we see that any admissible vector $\vec{\psi}_M$  can be written in the form
\begin{equation}
\label{E:GenBCsolFinal_psi}
\vec{\psi}_M = \vec{\Pi}_M \vec{\zeta},
\end{equation}
for some $\vec{\zeta}\in\mathbb{R}^{\dim[\mathcal{N}(\vec{K})]}$, where $\vec{\Pi}_M$ is a computable projection matrix. Note that in general $\vec{\Pi}_M$ may have linearly dependent columns and so it may be further simplified; this makes no difference to the following discussion, and we omit the details to streamline the presentation.
Substituting~\eqref{E:GenBCsolFinal_psi} into~\eqref{E:LowerBoundIntegralTerm}, we conclude that when~\eqref{E:CompoundLMI} holds, $\mathcal{F}_{\vec{\gamma}}\{\vec{w}\}$ is lower bounded over the space $H$ in~\eqref{E:HspaceDef} as
\begin{equation}
\label{E:LowerBoundFwithBC}
\mathcal{F}_{\vec{\gamma}}\{\vec{w}\} \geq 
\vec{\zeta}^T \,\vec{\Pi}_M^T\vec{Q}_M\vec{\Pi}_M\, \vec{\zeta} 
+ 
\int_{-1}^{1} \begin{bmatrix} U^k_{M} \\ V^k_{M} \end{bmatrix}^T 
[\vec{S}(x;\vec{\gamma}) 
- \vec{\Sigma}_M] 
\begin{bmatrix} U^k_{M} \\ V^k_{M} \end{bmatrix}
\,{\rm d}x.
\end{equation}

From~\eqref{E:newBCs} and~\eqref{E:GenBCsolFinal_psi} it is also possible to formulate a set of BCs that further restrict the choice for $U^k_{M}$ and $V^k_{M}$.  Moreover, recall that the remainder functions $U^k_{M}$ and $V^k_{M}$ should be orthogonal to all Legendre polynomials of degree less than or equal to $M$. However, it is not currently clear to the authors how these two constraints can be enforced explicitly in~\eqref{E:LowerBoundFwithBC} to obtain a stronger, but still useful, lower bound on $\mathcal{F}_{\vec{\gamma}}$. Consequently, we choose to simply drop them and let $U^k_{M}$ and $V^k_{M}$ be arbitrary functions.

\subsection{Formulating an inner SDP relaxation}
\label{S:MainResult}

The integral inequality in~\eqref{E:OptProblem} is satisfied if the right-hand side of~\eqref{E:LowerBoundFwithBC} is non-negative for all $\vec{\zeta}$ and all functions $U^k_{M}$ and $V^k_{M}$. 
Recalling that~\eqref{E:LowerBoundFwithBC} is valid only if~\eqref{E:CompoundLMI} holds, we have the following result.
\begin{proposition}
\label{T:InnerApprox}
Let $M=M(N)$ be as in~\eqref{E:Mdef} for any integer $N$, and let $\mathcal{Y}$ be as in~\eqref{E:AuxVarDef}. The set $T^\mathrm{in}_N\subset \mathbb{R}^s$  of values $\vec{\gamma}\in\mathbb{R}^s$ for which there exist $\mathcal{Y}$ such that
\begin{subequations}
\label{E:SuffCondAll}
\begin{align}
\label{E:SuffCond_a}
\overline{\vec{\Omega}}(\vec{\gamma};\mathcal{Y}) &\succeq 0,
\\
\label{E:SuffCond_b}
\vec{\Pi}^T_M \, \vec{Q}_{M}(\vec{\gamma},\mathcal{Y}) \,\vec{\Pi}_M 
&\succeq 0,
\\
\label{E:SuffCond_c}
\vec{S}(x;\vec{\gamma}) - \vec{\Sigma}_M(\vec{\gamma},\mathcal{Y}) &\geq 0, \quad\forall x\in[-1,1],
\end{align}
\end{subequations}
is an inner approximation of the feasible set $T$ of~\eqref{E:OptProblem}, i.e., $T^\mathrm{in}_N\subset T$.
\end{proposition}

Conditions~\eqref{E:SuffCond_b} and~\eqref{E:SuffCond_c} 
are only sufficient, not necessary, to make the right-hand side of~\eqref{E:LowerBoundFwithBC} non-negative: they do not take into account the boundary and orthogonality conditions on the remainder functions mentioned at the end of \S\ref{SS:BCproj}. However, they are useful because they can be turned into tractable constraints. 
For example,~\eqref{E:SuffCond_b} is not an LMI because $\vec{Q}_{M}(\vec{\gamma},\mathcal{Y})$ depends on absolute values of the Legendre coefficients of the entries of the matrix $\vec{F}(x;\vec{\gamma})$ in~\eqref{E:FmatDef} as a consequence of Lemma~\ref{T:LemmaR}.  However, it can readily be recast as one by replacing each of these absolute values, say $\vert\hat{f}_n(\vec{\gamma})\vert$, with a slack variable $t$ subject to the additional linear constraints $-t \leq \hat{f}_n(\vec{\gamma}) \leq t$~\cite{Boyd2004}.
Moreover, \eqref{E:SuffCond_c} is an LMI if the matrix $\vec{S}(x;\vec{\gamma})$ is independent of $x$, which is true in many interesting and non-trivial cases, such as our motivating example in \S\ref{S:ProblemDef}.
Otherwise, \eqref{E:SuffCond_c} is equivalent to the polynomial inequality
\begin{equation}
\label{E:Mat2Poly}
\vec{z}^T[\vec{S}(x;\vec{\gamma}) - \vec{\Sigma}_M(\vec{\gamma},\mathcal{Y})]\vec{z} \geq 0, \,\, \forall (x,\vec{z}) \in [-1,1]\times \mathbb{R}^2.
\end{equation}
Although checking a polynomial inequality is generally NP-hard (see~\cite[Sect. 2.1]{Parrilo2003} and references therein), we can turn~\eqref{E:Mat2Poly} into an LMI plus linear equality constraints by a SOS relaxation~\cite{Parrilo2003}. Using the so-called $\mathcal{S}$-procedure~\cite{Tan2006a}, we introduce a tunable symmetric polynomial matrix $\vec{T}(x)\in\mathbb{S}^2$ and require that the multivariate polynomials
\begin{align*}
p_1
&\defeq \vec{z}^T[\vec{S}(x;\vec{\gamma}) - \vec{\Sigma}_M(\vec{\gamma},\mathcal{Y}) - (1-x^2)\vec{T}(x)]\vec{z}, \\
p_2
&\defeq  \vec{z}^T\vec{T}(x)\vec{z},
\end{align*}
are SOS;
it is not difficult to see that this implies~\eqref{E:Mat2Poly}.

An upper bound for the optimal value of~\eqref{E:OptProblem}, as well as a feasible point that achieves it, can therefore be found by solving an SDP.
\begin{theorem}
\label{T:InnerSDPrelax}
Let $M=M(N)$ be defined as in~\eqref{E:Mdef} for any integer $N$, let $\mathcal{Y}$ be as in~\eqref{E:AuxVarDef}, and let $\vec{T}(x)\in\mathbb{S}^2$ be a tunable polynomial matrix. The optimal value of the SDP
\begin{align}
\label{E:InnerSDP}
&\qquad\qquad\qquad\quad\min_{\vec{\gamma},\mathcal{Y}, \vec{T}(x)} \quad \vec{c}^T \vec{\gamma}, 
\notag\\
\text{s.t.} \quad
&\overline{\vec{\Omega}}(\vec{\gamma};\mathcal{Y}) \succeq 0,
\\
&\vec{\Pi}^T_M \, \vec{Q}_{M}(\vec{\gamma},\mathcal{Y}) \,\vec{\Pi}_M 
\succeq 0,
\notag\\
&\vec{z}^T\left[\vec{S}(x;\vec{\gamma}) - \vec{\Sigma}_M(\vec{\gamma},\mathcal{Y}) - (1-x^2)\vec{T}(x)\right]\vec{z} \quad \text{is SOS},
\notag\\
&\vec{z}^T\vec{T}(x)\vec{z} \quad \text{is SOS},
\notag
\end{align}
is an upper bound for the optimal value of~\eqref{E:OptProblem}. Moreover, if  a minimizer $\vec{\gamma}^\star_N$ exists in~\eqref{E:InnerSDP}, it is a feasible point for~\eqref{E:OptProblem}.
\end{theorem}

\begin{remark}
\label{R:InnerFeasibilityRemark}
In contrast to our results for the outer SDP relaxations of \S\ref{S:OuterApprox}, we cannot prove that the optimal value of~\eqref{E:InnerSDP} converges to that of the original problem as $N$ is increased, nor that it is non-increasing. In fact, without further assumptions on the functional $\mathcal{F}_{\vec{\gamma}}$ in~\eqref{E:OptProblem}, it is possible that~\eqref{E:InnerSDP} is always infeasible even if~\eqref{E:OptProblem} is feasible. To see this, recall that the matrix $\vec{\Sigma}_M$ is positive definite, so~\eqref{E:SuffCond_c} and its corresponding SOS relaxation are feasible only if $\vec{S}(x;\vec{\gamma})$ can be made sufficiently positive definite for all $x\in[-1,1]$. An example for which this does not happen is the integral inequality
\begin{equation}
\label{E:NonUniformlyEllipticEx}
\int_{-1}^{1} \left[ x^2 (\partial u)^2 + (\partial v)^2 - \gamma uv \right]\,{\rm d}x \geq 0,
\end{equation}
where $u$ and $v$ are subject to the Dirichlet BCs $u\at{-1}=u\at{1}=v\at{-1}=v\at{1}=0$. 
This inequality is clearly feasible for $\gamma = 0$. Yet,~\eqref{E:InnerSDP} is infeasible for any $N$ because
$\vec{S}(x;\vec{\gamma})=\left[\begin{smallmatrix} x^2 & 0 \\ 0 & 1\end{smallmatrix}\right]$ is not positive definite at $x=0$. In fact, for this particular example \emph{any} approach requiring estimates of tail terms of series expansions will necessarily be ineffective. In contrast, with the
SOS method of~\cite{Valmorbida2016} we could establish that~\eqref{E:NonUniformlyEllipticEx} is feasible for
$\vert\gamma\vert  \leq 2.2$ at least.
With the exception of such pathological cases, however, our inner SDP relaxations are observed to work well in practice; we demonstrate this in \S\ref{S:Results}. This suggests that it may be possible to formulate precise conditions under which our inner SDP relaxations are feasible, and even converge to the original optimization problem. We leave this task to future research.
\end{remark}

\section{Extensions}
\label{S:Extensions}

\subsection{Inequalities with explicit dependence on boundary values}
\label{S:IneqBoundVal}

In the applications we have in mind, the integral inequality constraint in~\eqref{E:OptProblem} is derived from a weak formulation of a PDE, after integrating some terms by parts. Occasionally, the BCs are such that the boundary terms from such integrations by parts do not vanish; we will give an example in \S\ref{S:StabilitySystemPDEs}. 

This motivates us to extend our results to quadratic homogeneous functionals that depend explicitly on the boundary values $\mathcal{B}^{\vec{l}}\vec{w}$, 
such as
\begin{multline}
\label{E:FwithBCdep}
\mathcal{F}_{\vec{\gamma}}\{\vec{w}\}  \defeq
\int_{-1}^{1}\left[ 
\left(\mathcal{B}^{\vec{l}} \vec{w}\right)^T \, \F_\mathrm{bnd}(x;\vec{\gamma})\, \mathcal{B}^{\vec{l}} \vec{w}
+
\left(\mathcal{B}^{\vec{l}} \vec{w}\right)^T \,\F_\mathrm{mix}(x;\vec{\gamma})\, \mathcal{D}^{\vec{k}} \vec{w}
\right.
\\
\left.
+
\left(\mathcal{D}^{\vec{k}} \vec{w}\right)^T \,\F_\mathrm{int}(x;\vec{\gamma})\, \mathcal{D}^{\vec{k}} \vec{w}
\right]
\dx,
\end{multline}
where $\vec{F}_\mathrm{int}$, $\vec{F}_\mathrm{mix}$ and $\vec{F}_\mathrm{bnd}$ are matrices of polynomials of degree at most $d_F$ of the form~\eqref{E:FmatDef}. Note that $\mathcal{F}_{\vec{\gamma}}\{\vec{w}\}$ in~\eqref{E:FwithBCdep} reduces to the functional in~\eqref{E:OptProblem} if $\vec{F}_\mathrm{mix}=\vec{0}$, $\vec{F}_\mathrm{bnd}=\vec{0}$ and $\vec{F}_\mathrm{int} = \vec{F}$. 

The extension of Theorem~\ref{T:ConvergenceOuterSDPrelax} is obvious, because the boundary values of polynomial functions are easily given in terms of the polynomial coefficients.

To extend Proposition~\ref{T:InnerApprox} and Theorem~\ref{T:InnerSDPrelax}, we recall the definition of the permutation matrix $\vec{P}$ in~\eqref{E:PermMatDef}. Upon integrating the known matrix $\vec{P}^T\F_\mathrm{bnd}(x;\vec{\gamma})\vec{P}$, it follows from~\eqref{E:PsiDef} and Lemma~\ref{T:BCLemma}  that there exists a symmetric matrix $\vec{Q}^\mathrm{bnd}_M(\vec{\gamma})$ such that
\begin{equation}
\label{E:ExpansionFb}
\int_{-1}^{1} \left(\mathcal{B}^{\vec{l}}\vec{w}\right)^T
\F_\mathrm{bnd}(x;\vec{\gamma})
\,\mathcal{B}^{\vec{l}}\vec{w} \,{\rm d}x= \begin{bmatrix}
\vec{\psi}_{M}\\ \mathcal{B}^{[\vec{k},\vec{l}]}\vec{w}
\end{bmatrix}^T
\vec{Q}^\mathrm{bnd}_M(\vec{\gamma})
\begin{bmatrix}
\vec{\psi}_{M}\\ \mathcal{B}^{[\vec{k},\vec{l}]}\vec{w}
\end{bmatrix}.
\end{equation}

Moreover, let $\vec{g}(x;\vec{\gamma})$ be the column of the matrix $\vec{P}^T\F_\mathrm{mix}(x;\vec{\gamma})$ corresponding to the entry $\partial^\alpha u$ of $\mathcal{D}^{\vec{k}}\vec{w}$. Each element $g_i(x;\vec{\gamma})$ is a polynomial of degree at most $d_F$, written in the Legendre basis with coefficients $\hat{g}_{i,0}(\vec{\gamma}),\,\hdots,\,\hat{g}_{i,d_F}(\vec{\gamma})$. Recalling from~\eqref{E:Ncondition} that we have decomposed the Legendre expansion of $\partial^\alpha u$ with the truncation parameter $N\geq d_F+k-1$, we conclude that
\begin{align}
\label{E:ExpansionFmColumn}
\int_{-1}^{1} g_i(x;\vec{\gamma}) \partial^\alpha u \,{\rm d}x 
&= 
\sum_{m=0}^{d_F} \sum_{n=0}^{\infty} \hat{g}_{i,m}(\vec{\gamma}) \hat{u}^\alpha_n \int_{-1}^{1} \mathcal{L}_m\,\mathcal{L}_n\,{\rm d}x
\notag \\
&=
\begin{bmatrix}
2\hat{g}_{i,0}(\vec{\gamma}), & \displaystyle\frac{2\hat{g}_{i,1}(\vec{\gamma})}{3}, & \hdots, & 
\displaystyle\frac{2\hat{g}_{i,d_F}(\vec{\gamma})}{2d_F+1}\end{bmatrix}
\vec{\hat{u}}^\alpha_{[0,d_F]}.
\end{align}
With the help of Lemma~\ref{T:IntegrationLemma},~\eqref{E:PsiDef} and Lemma~\ref{T:BCLemma}
it is then possible to find a matrix
$\vec{Q}^\mathrm{mix}_M(\vec{\gamma})$ that satisfies
\begin{align}
\int_{-1}^{1} \left(\mathcal{B}^{\vec{l}}\vec{w}\right)^T \F_\mathrm{mix}(x;\vec{\gamma})\mathcal{D}^{\vec{k}}\vec{w}  \,{\rm d}x 
&=  
\begin{bmatrix} 
\mathcal{B}^{\vec{k-1}}\vec{w} \\ \mathcal{B}^{[\vec{k},\vec{l}]}\vec{w} 
\end{bmatrix}^T
\int_{-1}^{1} \vec{P}^T \F_\mathrm{mix}(x;\vec{\gamma})\mathcal{D}^{\vec{k}}\vec{w} \,{\rm d}x 
\notag \\
&=
\begin{bmatrix}
\vec{\psi}_{M}\\\mathcal{B}^{[\vec{k},\vec{l}]}\vec{w}
\end{bmatrix}^T
\vec{Q}^\mathrm{mix}_M(\vec{\gamma}) \,
\vec{\psi}_{M}.
\label{E:ExpansionFm}
\end{align}

Note that~\eqref{E:ExpansionFb} and~\eqref{E:ExpansionFm} are exact formulae, and no approximation is made. Combining these results with~\eqref{E:LowerBoundIntegralTerm}, we conclude that there is a symmetric matrix $\vec{Q}^\mathrm{tot}_M = \vec{Q}^\mathrm{tot}_M(\vec{\gamma},\mathcal{Y})$ such that
\begin{equation}
\label{E:FboundExt}
\mathcal{F}_{\vec{\gamma}}\{\vec{w}\} \geq 
\begin{bmatrix}
\vec{\psi}_{M}\\\mathcal{B}^{[\vec{k},\vec{l}]}\vec{w}
\end{bmatrix}^T \vec{Q}^\mathrm{tot}_M \begin{bmatrix}
\vec{\psi}_{M}\\\mathcal{B}^{[\vec{k},\vec{l}]}\vec{w}
\end{bmatrix} 
+ 
\int_{-1}^{1} \begin{bmatrix} U^k_{M} \\ V^k_{M} \end{bmatrix}^T 
[\vec{S}(x;\vec{\gamma}) - \vec{\Sigma}_M] 
\begin{bmatrix} U^k_{M} \\ V^k_{M} \end{bmatrix}
\,{\rm d}x.
\end{equation}
Finally,~\eqref{E:newBCs} implies that we can write
\begin{equation*}
\begin{bmatrix} \vec{\psi}_M \\ \mathcal{B}^{[\vec{k},\vec{l}]}\vec{w} \end{bmatrix} = \vec{\Lambda} \vec{\zeta}
\end{equation*}
for some $\vec{\zeta}\in\mathbb{R}^{\dim[\mathcal{N}(\vec{K})]}$, where the projection matrix $\vec{\Lambda}$ satisfies $\mathcal{R}(\vec{\Lambda})=\mathcal{N}(\vec{K})$, and we conclude that Proposition~\ref{T:InnerApprox} and Theorem~\ref{T:InnerSDPrelax} are true when we replace~\eqref{E:SuffCond_b} and the corresponding constraint in~\eqref{E:InnerSDP} with 
$\vec{\Lambda}^T
\vec{Q}^\mathrm{tot}_M(\vec{\gamma})
\vec{\Lambda}
\succeq 0.
$

\subsection{Higher-dimensional function spaces \& generic multi-index derivatives}

Theorems~\ref{T:ConvergenceOuterSDPrelax} and~\ref{T:InnerSDPrelax} were derived with the assumption that $\vec{w}\in C^m([-1,1],\mathbb{R}^2)$ and for the particular multi-indices $\vec{k}=[k,k]$, $\vec{l}=[l,l]$. All our statements, including the extensions discussed in \S\ref{S:IneqBoundVal}, hold also when we let $\vec{w}\in C^m([-1,1],\mathbb{R}^q)$ with $q\geq 1$ and when $\vec{k},\,\vec{l}\in\mathbb{N}^q$ are generic multi-indices, as long as they satisfy~\eqref{E:klCond1} and~\eqref{E:klCond2}.

In particular, all our proofs extend verbatim by simply identifying the functions $u,\,v$ used throughout \S\ref{S:OuterApprox} and \S\ref{S:InnerApprox} with any two components $w_i,\,w_j$ of $\vec{w}$ if the $q$-dimensional multi-indices $\vec{k}$ and $\vec{l}$ are uniform, i.e., $\vec{k}=[k,\,\hdots,\,k]$ and $\vec{l}=[l,\,\hdots,\,l]$. The extension to non-uniform multi-indices $\vec{k},\,\vec{l}\in\mathbb{N}^q$ requires only minor modifications; the details are left to the interested reader.

\section{Computational experiments with \textsc{QUINOPT}}
\label{S:Results}

In this section we apply our techniques to solve some problems arising from the analysis of PDEs. To aid the formulation of our SDP relaxations, we have developed
\textsc{QUINOPT} (QUadratic INtegral OPTimization), an open-source add-on for the \textsc{MATLAB} optimization toolbox \textsc{YALMIP}~\cite{Lofberg2004,Lofberg2009}. \textsc{QUINOPT} uses the Legendre polynomial basis for the outer SDP relaxations of \S\ref{S:OuterApprox}, because the orthogonality of the Legendre polynomials promotes sparsity of the SDP data. \textsc{QUINOPT} and the scripts used to produce the results in the following sections can be downloaded from 
\begin{center}
\url{https://github.com/aeroimperial-optimization/QUINOPT}. 
\end{center}
Our experiments were run on a PC with a 3.40GHz Intel\textsuperscript{\textregistered} Core{\texttrademark} i7-4770 CPU and 16Gb of RAM, using \textsc{MOSEK}~\cite{Andersen2009} to solve our SDP relaxations.

\subsection{Motivating example: stability of a stress-driven shear flow}
\label{S:ShearFlowEx}

Consider our motivating example of \S\ref{S:MotivatingExample}, regarding the stability of a flow driven by a shear stress of magnitude $0.5\gamma$. An ad-hoc inner SDP relaxation was proposed and solved in~\cite{Fantuzzi2016}; here, we replicate those results using our general-purpose toolbox \textsc{QUINOPT}. 
Since we minimize the \emph{negative} of $\gamma$ in~\eqref{E:ProbEx2}, the inner and outer SDPs~\eqref{E:InnerSDP} and~\eqref{E:OuterSDP} give, respectively, lower and upper bounds for the stress $\gamma_\mathrm{cr}$ at which the flow is no longer provably stable.

Figure~\ref{F:ShearFlowSol} shows the upper and lower bounds for $\gamma_\mathrm{cr}$ as a function of the wave number $\xi$, a parameter in~\eqref{E:ProbEx2}, computed for four different values of the Legendre series truncation parameter $N$. No upper bound curve is plotted for $N=3$ because in this case only the zero polynomial satisfies the BCs in~\eqref{E:Prob2BCs}, and~\eqref{E:OuterSDP} reduces to an unconstrained minimization problem yielding an infinite upper bound.
More detailed numerical results, CPU times, and the number of primal and dual variables in the SDP relaxations (denoted $n$ and $m$ respectively) are reported in Table~\ref{T:TableShearFlowEx} for 
wave number parameters $\xi=3$ and $\xi=9$. For comparison, Table~\ref{T:TableShearFlowEx2} gives lower bounds on $\gamma_\mathrm{cr}$ computed with the inner SOS relaxation method of~\cite{Valmorbida2016} using polynomials of degree $d$, as well as the primal-dual dimensions of the corresponding SDPs returned by YALMIP's SOS module~\cite{Lofberg2009} and the CPU time required to solve them on our machine.

\begin{figure}[t]
\centering
\includegraphics[scale=0.86]{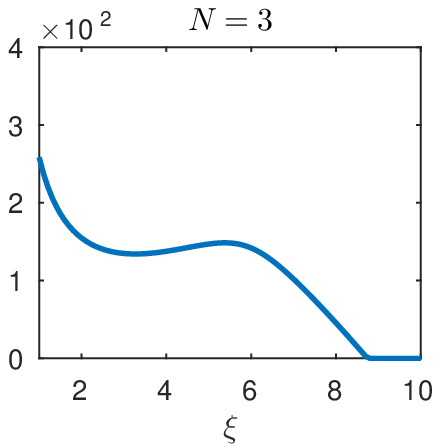} \hfill
\includegraphics[scale=0.86]{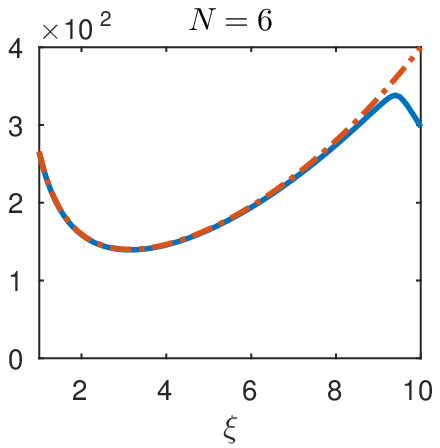} \hfill
\includegraphics[scale=0.86]{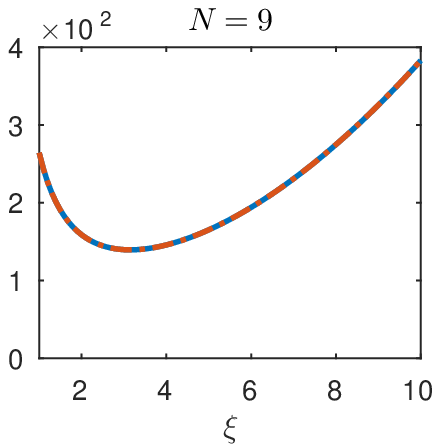} \hfill
\includegraphics[scale=0.86]{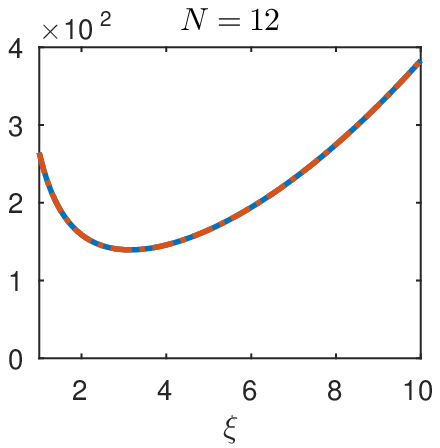}
\caption{(Color online) Upper (dot-dashed line) and lower (solid line) bounds on the optimal value of~\eqref{E:ProbEx2} as a function of $\xi$ for different values of $N$. The upper bound for $N=3$ is infinite and so it is not plotted. The bounds are indistinguishable for $N=9$ and $N=12$.}
\label{F:ShearFlowSol}
\end{figure}

\begin{table}[b]
\caption{Upper and lower bounds on $\gamma_\mathrm{cr}$ (denoted LB and UB), CPU times (in seconds), and primal-dual problem dimensions ($n$ and $m$) for the outer and inner relaxations of problem~\eqref{E:ProbEx2} from \textsc{QUINOPT}, as a function of the Legendre truncation parameter $N$.}
\centering
\small
\begin{tabular}{cc rrrrr cc rrrrr}
\hline\\[-0.75em]
&& \multicolumn{5}{c}{QUINOPT, outer} &&
& \multicolumn{5}{c}{QUINOPT, inner}\\
\cline{3-7} \cline{10-14}\\[-0.75em]
 && $N$ & $n$ & $m$ & UB & $t$ 
&&& $N$ & $n$ & $m$ & LB & $t$
\\[0.25em]
\multirow{4}{*}{$\xi=3$} 	   
	&& 3 & 0 & 1 & +INF &   0.03
	&&& 3 & 202 & 2 & 134.8594 & 0.09\\
	&& 6 & 36 & 1 & 140.4087 &   0.04 
	&&& 6 & 406 & 2 & 139.7656 &   0.10 \\
	&& 9 & 144 & 1 &  139.7701 &   0.06 
	&&& 9 & 683 & 2 & 139.7700 &   0.08 \\
	&& 12 & 324 & 1 & 139.7700 &   0.05 
	&&& 12 & 1030 & 2 & 139.7700 &   0.10 \\[0.5em]
\multirow{4}{*}{$\xi=9$} 	 
	&& 3 &   0 & 1 &   +INF &   0.03 
	&&& 3 & 202 & 2 & 0.0000 &   0.08\\
	&& 6 &  36 & 1 & 335.1022 &   0.04 
	&&& 6 & 406 & 2 & 323.5764 &   0.08 \\
	&& 9 & 144 & 1 & 325.6764 &   0.05
	&&& 9 & 683 & 2 & 325.6449 &   0.09  \\
	&& 12 & 324 & 1 & 325.6455 &   0.05 
	&&& 12 & 1030 & 2 & 325.6453 &   0.10\\[0.25em]
\hline
\end{tabular}
\label{T:TableShearFlowEx}
\end{table}

\begin{table}[t]
\caption{Lower bounds on $\gamma_\mathrm{cr}$ (denoted LB), CPU times (in seconds), and primal-dual problem dimensions ($n$ and $m$) for the inner relaxations of problem~\eqref{E:ProbEx2} obtained with the SOS method of~\cite{Valmorbida2016} using polynomials of degree $d$.}
\centering
\small
\begin{tabular}{rrrrr cc rrrrr}
\hline\\[-0.75em]
\multicolumn{5}{c}{$\xi=3$} &&&
\multicolumn{5}{c}{$\xi=9$}
\\
\cline{1-5} \cline{8-12}\\[-0.75em]
$d$ & $n$ & $m$ & LB & $t$ &&&
$d$ & $n$ & $m$ & LB & $t$
\\[0.25em]	   
   4 &  805 &   230 &  79.4435 & 0.19 &&
&  4 &   805 &   230 & 285.9021 & 0.18\\
   8 &  2349 &   454 & 119.8619 & 0.28 &&
&  8 &  2349 &   454 & 314.1146 & 0.27\\
  16 &  7789 &   902 & 130.5796 & 0.68 &&
& 16 &  7789 &   902 & 321.2403 & 0.65\\
  32 & 28077 &  1798 & 134.4737 & 3.16 &&
& 32 & 28077 &  1798 & 323.1421 & 2.98\\[0.25em]
\hline
\end{tabular}
\label{T:TableShearFlowEx2}
\end{table}

Our results show that within the tested range of $\xi$ the upper and lower bounds converge to each other at relatively small values of $N$ (three decimal places for $N=12$ for both cases reported in Table~\ref{T:TableShearFlowEx}). This means that we can bound $\gamma_\mathrm{cr}$ accurately and extremely efficiently using our techniques, and that the inner SDP relaxations converge to the full problem~\eqref{E:ProbEx2} despite our inability to provide a proof of this fact in general (cf. Remark~\ref{R:InnerFeasibilityRemark}). Finally, note that our techniques significantly outperform the SOS method of~\cite{Valmorbida2016} in terms of computational cost and quality of the lower bound.

\subsection{Stability of a system of coupled PDEs}
\label{S:StabilitySystemPDEs}

Let $\vec{w}=\left[u(t,x),\,v(t,x)\right]^T$  
and consider the system of PDEs
\begin{equation}
\label{E:PDEsystemEx}
\partial_t \vec{w} = \gamma\partial^2_x \vec{w}+ \vec{A}\vec{w}, \qquad 
\vec{A}=\begin{bmatrix}
1 & 1.5 \\ 5 & 0.2
\end{bmatrix},
\end{equation}
over the domain $[0,1]$, subject to the BCs $u\at{0}=u\at{1}=v\at{0}=v\at{1}=0$. This system was studied in~\cite[Sect. V-D]{Valmorbida2014a} with the equivalent parametrization $\gamma=R^{-1}$. The stabilizing effect of the diffusive term $\gamma\partial^2_x \vec{w}$ decreases with $\gamma$, until the equilibrium solution $[u,v]^T=[0,0]^T$ becomes unstable. It can be shown that the amplitude of infinitesimal sinusoidal perturbations to the zero solution grows exponentially in time if $\gamma < \gamma_\mathrm{cr}=0.3412$. Since the system is linear, it is stable with respect to finite-amplitude perturbations for all $\gamma \geq \gamma_\mathrm{cr}$. 

Following~\cite{Valmorbida2014a}, we try to establish the stability of the system with respect to arbitrary perturbations by considering Lyapunov functionals of the form
\begin{equation}
\label{E:PDELyapFun}
\mathcal{V}(t) = \frac{1}{2}\int_{0}^{1}\vec{w}^T \vec{P}(x) \vec{w}\dx,
\end{equation}
where $\vec{P}(x)$ is a tunable polynomial matrix of given degree $d_P$, such that $\mathcal{V}(t)\geq c \norm{\vec{w}}_2^2$ for some $c>0$ and $-\frac{{\rm d}\mathcal{V}}{{\rm d}t}\geq 0$. Note that since $\vec{P}(x)$ can always be rescaled by $c$ without changing the sign of the inequalities, we may fix $c=1$.  

Using~\eqref{E:PDEsystemEx} to compute $\frac{{\rm d}\mathcal{V}}{{\rm d}t}$, we find that the critical value of $\gamma$ at which~\eqref{E:PDELyapFun} stops being a valid Lyapunov function for a given degree $d_P$ is given by
\begin{equation}
\label{E:FeasProblemPDEEx}
\begin{gathered}
\min_{\gamma,\vec{P}(x)} \quad \gamma \\
\begin{aligned}
\text{s.t.} \quad
&\int_{0}^{1} \vec{w}^T \left[ \vec{P}(x) - \vec{I}\right] \vec{w} \,{\rm d}x \geq 0,
\\&
\int_{0}^{1} \vec{w}^T \vec{P}(x) (-\gamma\partial^2_x \vec{w}-\vec{A}\vec{w}) \dx \geq 0.
\end{aligned}
\end{gathered}
\end{equation}
Note that although the system state $\vec{w}$ is a function of time, the integral inequalities above are imposed {\it pointwise in time}. Therefore, the time dependence can be formally dropped, and~\eqref{E:FeasProblemPDEEx} is in the form~\eqref{E:OptProblem} with two integral inequalities.

Since the optimization variables are $\gamma$ and the coefficients of the entries of $\vec{P}(x)$, the problem is not jointly convex in $\gamma$ and $\vec{P}$, and we cannot minimize $\gamma$ directly. Instead, we fix a trial value for $\gamma$ and check whether a feasible $\vec{P}(x)$ of degree $d_P$ exists. The optimal $\gamma$ for~\eqref{E:FeasProblemPDEEx}, which must finite because the system is linearly unstable when $\gamma$ is sufficiently small, is then given by the value at which a feasible $\vec{P}(x)$ ceases to exist, and it can be determined with a simple bisection procedure. 

Before deriving our SDP relaxations, we need to rescale the domain of integration for the constraints in~\eqref{E:FeasProblemPDEEx} to $[-1,1]$. Moreover, in light of Remark~\ref{R:InnerFeasibilityRemark}, the second integral inequality should be integrated by parts to prevent the inner SDP relaxation from being infeasible. Both tasks (rescaling and integration by parts) are performed automatically by \textsc{QUINOPT}. We also note that after rescaling and integration by parts the second integral inequality in~\eqref{E:FeasProblemPDEEx} depends explicitly on the unspecified boundary values $\partial_x u\at{\pm1}$ and $\partial_x v\at{\pm1}$, making the extensions discussed in \S\ref{S:IneqBoundVal} necessary.

\begin{table}[b]
\caption{Upper and lower bounds (UB and LB) for the optimal solution of~\eqref{E:FeasProblemPDEEx} for different choices of $d_P$, and for the case $\vec{P}(x)=\vec{I}$. Also reported are the average CPU times $t_{UB}$ and $t_{LB}$ to solve each feasibility problem in the bisection procedure to compute the upper/lower bounds UB and LB (to minimize $\gamma$ in the case $\vec{P}(x)=\vec{I}$), and upper bounds from~\cite{Valmorbida2014a}.}
\centering
\small
\begin{tabular}{c c cc cc}
\hline\\[-0.75em]
$d_P$ 
& UB from~\cite{Valmorbida2014} 
& UB & $t_\mathrm{UB}$ 
& LB & $t_\mathrm{LB}$ \\[0.5em]
$\vec{P}(x)=\vec{I}$ & 5 & 0.3925 & 0.26 & 0.3925 & 0.09 \\
0 & 3.3333 & 0.3412 & 0.14 & 0.3412 & 0.12 \\ 
2 & 0.5882 & 0.3412 & 1.32 & 0.3412 & 0.99 \\
4 & 0.4347 & 0.3412 & 1.57 & 0.3412 & 1.07 \\
6 & 0.4166 & 0.3412 & 1.82 & 0.3412 & 1.18 \\[0.25em]
\hline
\end{tabular}
\label{T:TablePDESystemEx}
\end{table}

Table~\ref{T:TablePDESystemEx} shows upper and lower bounds for the optimal solution of~\eqref{E:FeasProblemPDEEx} as a function of the degree $d_P$ of $\vec{P}(x)$, obtained by applying the bisection procedure described above to the SDPs~\eqref{E:InnerSDP} and~\eqref{E:OuterSDP} respectively. We also show results for the particular choice $\vec{P}(x)=\vec{I}$, corresponding to the classical approach of taking the energy of the system as the candidate Lyapunov function; in this case, a direct minimization over $\gamma$ could be performed. In all computations we fixed the Legendre series truncation parameter to $N=10$ and the degree of the matrix $\vec{T}(x)$ in~\eqref{E:InnerSDP} to $6$, which gives well converged results. Table~\ref{T:TablePDESystemEx} also reports the average CPU time taken by \textsc{QUINOPT} to set up and solve each feasibility problem in our bisection procedure (to minimize $\gamma$ when we fixed $\vec{P}(x)=\vec{I}$). 

Our results show that stability can be established up to the known critical value $\gamma_\mathrm{cr}=0.3412$ for all choices of $d_P$, with the exception of the classical energy Lyapunov function. This drastically improves the conservative results obtained with the SOS method in~\cite{Valmorbida2014a} for the same problem, also reported in Table~\ref{T:TablePDESystemEx} (the original results are for a parameter $R=\gamma^{-1}$ and have been adapted). Our results demonstrate that our SDP relaxations accurately approximate~\eqref{E:FeasProblemPDEEx}; this is particularly significant for the inner SDPs, which rely on typically conservative estimates and for which we cannot prove convergence.

\subsection{Feasible set approximation}
\label{S:FeasSEtApproxEx}

In this final example, we consider the problem of computing the entire feasible set of the integral inequality 
\begin{equation}
\label{E:ToyProblem}
\int_{-1}^{1} \left[ \left(\partial u\right)^2 + \left(\partial v\right)^2 + \gamma_1 x^2 \partial u\,\partial v+ 2\gamma_2 u v \right] \dx \geq 0,
\end{equation}
where $u$ and $v$ are subject to the Dirichlet BCs
$u\at{-1}=0$, $u\at{1}=0$, $v\at{-1}=0$, $v\at{1}=0$.
This inequality does not arise from a particular PDE, but has been constructed ad-hoc to illustrate some subtle properties of our SDP relaxations and highlight the main sources of conservativeness.

Outer and inner approximation sets $T^\mathrm{out}_N$ and $T^\mathrm{in}_N$ can be found using~\eqref{E:OuterSDP} and~\eqref{E:InnerSDP}, respectively. In particular, we compute the boundaries of $T^\mathrm{out}_N$ and $T^\mathrm{in}_N$ by optimizing the objective function $\gamma_1 \sin\theta + \gamma_2\cos\theta$ for 300 equispaced values of $\theta\in [0,2\pi]$. When solving~\eqref{E:InnerSDP}, we fix the degree of the tunable polynomial matrix $\vec{T}(x)$ to the smallest of $N-2$ and $6$; our results do not improve when this is increased. Inner approximation sets $T^\mathrm{sos}_N$ can be computed in a similar way using the SOS method of~\cite{Valmorbida2016} with polynomials of degree $N$.

The CPU time required to compute $T^\mathrm{out}_N$, $T^\mathrm{in}_N$ , and $T^\mathrm{sos}_N$  is shown in Table~\ref{T:FeasibleSetsTable} for six values of $N$ and two SDP solvers, \textsc{MOSEK}~\cite{Andersen2009} and \textsc{SDPT3}~\cite{Tutuncu2003}; $N=2$ is the minimum value that satisfies~\eqref{E:Ncondition}. Evidently, the SOS method is much more computationally expensive than our methods for high-degree relaxations. Rather surprisingly, \textsc{MOSEK} computes $T^\mathrm{in}_N$ more efficiently than $T^\mathrm{out}_N$ at large $N$, despite the latter being nominally cheaper; this is not the case for \textsc{SDPT3}. 

On the other hand, Figure~\ref{F:FeasibleSets} shows that while $T^\mathrm{sos}_N$ seems to converge to $T^\mathrm{out}_N$ as $N$ increases, the inner approximation sets $T^\mathrm{in}_N$ computed with the method of \S\ref{S:InnerApprox} do not: our estimates in Lemmas~\ref{T:LemmaQ} and~\ref{T:LemmaR} and the SOS relaxation of the polynomial matrix inequality~\eqref{E:SuffCond_c} introduce conservativeness. 

\begin{table}[b]
\caption{CPU time (in seconds) for the computation of the sets $T^\mathrm{out}_N$, $T^\mathrm{in}_N$, and $T^\mathrm{sos}_N$ as a function of $N$ using MOSEK~\cite{Andersen2009} and SDPT3~\cite{Tutuncu2003}.}
\centering
\small
\begin{tabular}{c c ccc c ccc}
\hline\\[-0.75em]
&& \multicolumn{3}{c}{\textsc{MOSEK}}
&& \multicolumn{3}{c}{\textsc{SDPT3}}\\
\cline{3-5} \cline{7-9}\\[-0.75em] 
$N$
&& $T^\mathrm{out}_N$ & $T^\mathrm{in}_N$ & $T^\mathrm{sos}_N$
&& $T^\mathrm{out}_N$ & $T^\mathrm{in}_N$ & $T^\mathrm{sos}_N$\\[0.5em]
2 && 0.81 & 1.84 & 1.58 && 10.5 & 25.4 & 21.5\\
4 && 0.95 & 2.36 & 4.10 && 12.6 & 30.9 & 45.1\\
8 && 1.66 & 4.99 & 19.2 && 14.5 & 46.4 & 176\\
16 && 4.82 & 6.80 & 346 && 21.1 & 58.2 & 600\\
24 && 9.36 & 9.11 & 	2100  && 25.5 & 59.0 & 10500\\
32 && 17.7 & 14.6 & 6220 && 35.5 & 71.5 & 117000\\[0.25em]
\hline
\end{tabular}
\label{T:FeasibleSetsTable}
\end{table}

\begin{figure}[t]
\centering
\includegraphics[scale=1,trim = 0cm 0cm 0cm 0.1cm]{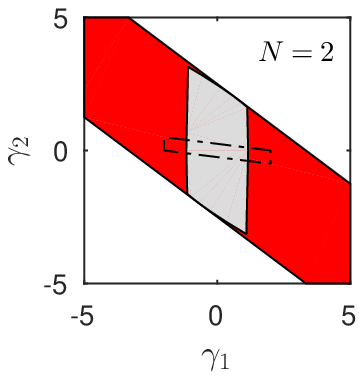} \hspace{15pt}
\includegraphics[scale=1,trim = 0cm 0cm 0cm 0.1cm]{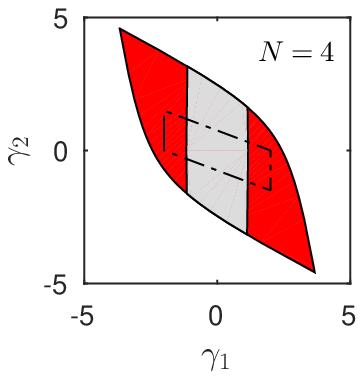} \hspace{15pt}
\includegraphics[scale=1,trim = 0cm 0cm 0cm 0cm]{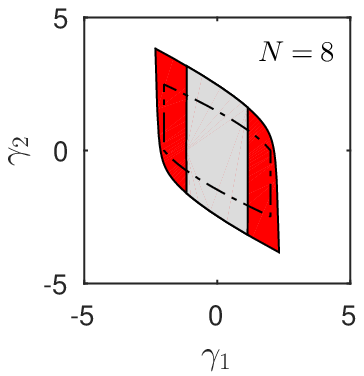} \\[0.25em]
\includegraphics[scale=1,trim = 0cm 0cm 0cm 0cm]{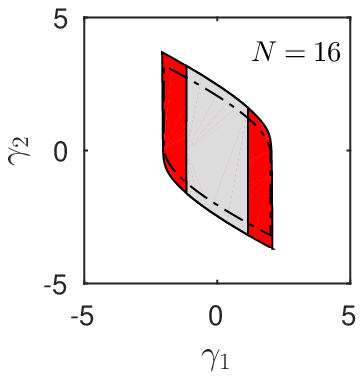} \hspace{15pt}
\includegraphics[scale=1,trim = 0cm 0.1cm 0cm 0cm]{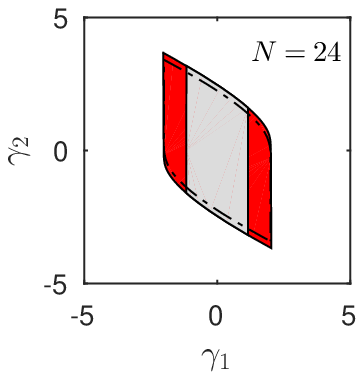}\hspace{15pt}
\includegraphics[scale=1,trim = 0cm 0.1cm 0cm 0cm]{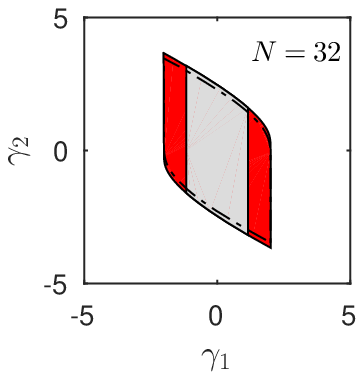}
\caption{(Color online) Inner and outer approximations of the feasible set of~\eqref{E:ToyProblem}: $T^\mathrm{in}_N$ (black solid boundary, gray interior), $T^\mathrm{out}_N$ (black solid boundary, red interior), and $T^\mathrm{sos}_N$ (dot-dashed black boundary, no shading).}
\label{F:FeasibleSets}
\end{figure}

Yet, there are parts where of the boundaries of $T^\mathrm{out}_N$ and $T^\mathrm{in}_N$ almost coincide even for $N$ as low as $4$, and---rather interestingly---these corresponds to those regions convergence of $T^\mathrm{sos}_N$ to $T^\mathrm{out}_N$ is the slowest. The figures suggest that the inner approximation sets $T^\mathrm{in}_N$ are only over-constrained in the $\gamma_1$ direction. This is because $\gamma_2$ appears only in the term $\int_{-1}^{1}2 \gamma_2 u v \dx$, to which we apply the estimates in Lemma~\ref{T:LemmaR} when computing the inner SDP relaxation. According to the decay rates stated in the Lemma, however, these estimates become negligible at large $N$. On the contrary, $\gamma_1$ appears in the term $\int_{-1}^{1}\gamma_1 x^2 \partial u\,\partial v \,{\rm d}x$, to which the estimates in part (ii) of Lemma~\ref{T:LemmaQ} must be applied. Despite our efforts to tune the auxiliary matrices in~\eqref{E:LemmaQEstimate}, the magnitude of such estimates does not decay compared to other terms, limiting the range of feasible values of $\gamma_1$ in practice.  
This issue should be addressed in future work, and should be taken into account when trying to formulate rigorous statements on the feasibility and convergence of our inner SDP relaxations.

\section{Scalability}
\label{S:scalability}

It may be checked that when $\vec{w}(x)\in\mathbb{R}^q$ is subject to $p$ independent boundary conditions, the degree of the polynomials in the matrix $\vec{F}(x;\vec{\gamma})$ is at most $d_F$, and $\vec{\gamma}\in\mathbb{R}^s$, the outer SDP relaxation for a quadratic inequality of the form~\eqref{E:FwithBCdep} with $N$ Legendre coefficients results in an SDP with an LMI of dimension $q(N+1)-p$ with $s$ variables. Instead, the inner SDP relaxation has:
an LMI of dimension $2\vert\vec{l}\vert + q(N+\vert\vec{k}\vert_\infty+ d_F+2)-p$, where $\vert\vec{k}\vert_\infty := \max_{i\in\{1,\,\ldots,\,q\}}k_i$;
 a $q\times q$ matrix SOS constraint of degree $\deg \vec{S}(x;\vec{\gamma})$, where $\vec{S}(x;\vec{\gamma})$ is defined as in \S\ref{SS:LowerBoundF};
at most $q(q+1)/2$ auxiliary LMIs of size $4\,d_F$ from Lemma~\ref{T:LemmaQ};
at most $(d_F+1)(2\,q+\vert\vec{k}\vert+1)\vert\vec{k}\vert$ linear inequalities to lift the absolute values introduced by Lemma~\ref{T:LemmaR};
at most $s+q(q+1)(2\,d_F^2+d_F+2)/2 + (d_F+1)(2\,q+\vert\vec{k}\vert+1)\vert\vec{k}\vert/2$ variables.
Since currently only small to medium-size SDPs can be solved in practice, one might therefore expect that although our techniques are cheaper than the SOS method of~\cite{Valmorbida2016}---as highlighted by our numerical examples---they can only be implemented when $q$, $d_F$, $s$ and $\vert\vec{k}\vert_\infty$ are sufficiently small. 

The development of solvers for large scale SDPs is an active research area, and new tools are being developed that should facilitate solving problems at larger scales; see, for example, the solvers SCS~\cite{ODonoghue2016} and CDCS~\cite{Zheng2016a,Zheng2016b}.

Moreover, the poor scalability of SDPs may not be too severe an issue for many problems of practical interest.
In fact, the number of constraints in the inner SDP relaxation can be considerably smaller than the worst-case count presented above. To see this, note that Lemma~\ref{T:LemmaQ} introduces auxiliary LMIs and variables only for the (upper-triangular, by symmetry) entries of $\vec{S}(x;\vec{\gamma})$ that depend on $x$; for example, only one auxiliary LMI is needed for inequality~\eqref{E:ToyProblem}. In addition, the size of the auxiliary LMI associated with the entry $\vec{S}_{ij}$ can be reduced to $4\times\deg \vec{S}_{ij}$, yielding considerable savings if $\deg \vec{S}(x;\vec{\gamma})\ll d_F$. In the extreme case $\deg \vec{S}(x;\vec{\gamma})=0$, i.e. the matrix $\vec{S}(x;\vec{\gamma})$ is independent of $x$, there are no auxiliary variables and LMIs from Lemma~\ref{T:LemmaQ}, and moreover the $q\times q$ matrix SOS constraint becomes a $q\times q$ LMI. 
This situation is common when energy-Lyapunov-function methods are applied to turbulent fluid flows~\cite{Constantin1995a,Doering1994,Doering1996}, so our techniques are particularly suited to tackle problems in this field---as proven by the results of \S\ref{S:ShearFlowEx} and of~\cite{Fantuzzi2016}.

Finally, we also note that a moderate Legendre truncation parameter $N$, and hence a medium-size SDP relaxation, often suffices to obtain accurate bounds on the objective function, as suggested by all our examples. Roughly speaking, to obtain a good bound on the optimal value one should choose $N$ such that the minimizer $\vec{w}^\star$ of $\mathcal{F}_{\vec{\gamma}}\{\vec{w}\}$ at the optimal point $\vec{\gamma} = \vec{\gamma}^\star$ is approximated sufficiently well by a polynomial of degree $N$ (here we assume that the minimizer $\vec{w}^\star$ exists for simplicity). Since $\vec{w}^\star$ is typically a ``well-behaved'' function 
(the highest-order derivatives of highly oscillatory test functions $\vec{w}$ would give large contribution to $\mathcal{F}_{\vec{\gamma}^\star}\{\vec{w}\}$, making highly-oscillatory minimizers unlikely), this can be done with moderate $N$.

\section{Conclusion}
\label{S:Conclusion}

In this work, we have developed a new method to optimize a linear cost function subject to homogeneous quadratic integral inequality constraints. More precisely, we have employed Legendre series expansions and functional estimates to derive inner and outer approximations of the feasible set of an integral inequality, and have shown that upper and lower bounds for the optimal cost value can be computed efficiently using semidefinite programming. We have proven that the lower bounds obtained with our outer approximations form a non-decreasing sequence that converges to the exact optimal cost value (if this is attained). Unfortunately, similar statements do not generally extend to our inner approximations.

Although the steps leading to our SDP relaxations are rather technical, they are amenable to numerical implementation. To aid the formulation and solution of optimization problem with integral inequality constraints in practice, we have developed the \textsc{MATLAB} package \textsc{QUINOPT}, an open-source add-on for the optimization toolbox \textsc{YALMIP}. Using this software, we have successfully solved non-trivial problems that arise when studying the stability of autonomous systems of PDEs.

We have demonstrated that our methods work well in practice, even though they rely on typically conservative estimates to formulate numerically tractable constraints. It is in the interest of future work to formalize these observations, and determine conditions that ensure the feasibility and/or convergence of our inner SDP relaxations. The results presented in \S\ref{S:FeasSEtApproxEx} suggest that more stringent assumption on the properties of the integral inequality might be needed.

Looking at the applications we have in mind, i.e., the analysis of systems governed  by PDEs, the present work should be extended to {\it(i)} integral inequalities with explicit time dependence that arise from non-autonomous PDEs, and {\it(ii)}  inequalities over two or higher dimensional domains. Polynomial explicit time dependence could be dealt with by relaxing our inner/outer LMI constraints, now time-dependent, into matrix SOS conditions, although the (current) poor scalability of SOS optimization makes this strategy unlikely implementable. Multi-dimensional compact box domains could be analyzed by introducing Legendre expansions in each coordinate direction and adapting the ideas presented in this work, while for more general domains---including the non-compact case---other basis functions could be used. 
This may present hurdles in the derivation of inner approximations, because they require estimates that rely on specific properties of the basis functions. 
Unless sparsity and/or problem structure are exploited, multi-dimensional inequalities are also likely to be constrained by the current computational limitations: with $n$ spatial dimensions and $q$ dependent variables ($\vec{w}\in\mathbb{R}^q$), the LMI size for a simple outer approximations using polynomials of degree $N$ will be approximately $q N^n$.

Finally, it is in the interest of future work to extend our methods to integral inequalities more general than the homogeneous quadratic type. We expect that our methods can be extended with little effort to complete (i.e., inhomogeneous) quadratic integral inequalities over spaces described by homogeneous BCs (inhomogeneous BCs can be ``lifted'' by a polynomial shift). In fact, the linear part of a complete quadratic functional can be analyzed with ideas similar to those used in \S\ref{S:IneqBoundVal}. Extensions to higher-than-quadratic functionals, e.g. by introducing additional slack variables to reduce them to quadratic ones, are also essential if recently developed analysis techniques based on dissipation inequalities~\cite{Ahmadi2016} are to be successfully applied to complex nonlinear systems of PDEs of interest in physics and engineering.

\appendix
\section{Legendre polynomials and Legendre series}
\label{A:LegPolys}

The Legendre polynomial of degree $n$ is defined over the interval $[-1,1]$ as
\begin{equation*}
\mathcal{L}_n(x) = \frac{1}{n!\,2^n} \frac{{\rm d}^n}{\dx^n}(x^2-1)^n.
\end{equation*}
The Legendre polynomials of degree $n\geq 2$ can also be constructed with the recurrence relation
\begin{equation}
\label{E:LegpPolyRecurrence}
{n}\mathcal{L}_{n}(x) = \left({2n-1}\right)x\,\mathcal{L}_{n-1}(x) 
- {(n-1)}\mathcal{L}_{n-2}(x),
\end{equation}
with $\mathcal{L}_0(x) = 1$ and $\mathcal{L}_1(x) = x$. Equation~\eqref{E:LegpPolyRecurrence} can be used to show that $\mathcal{L}_n(\pm1) = (\pm1)^n$.

The Legendre polynomials satisfy a number of other recurrence relations. In this work, we will use the fact that
\begin{equation}
\label{E:LegRecursionDerivatives}
(2n+1)\mathcal{L}_n(x) = \frac{{\rm d}}{ \dx }\left[ \mathcal{L}_{n+1}(x) - 
\mathcal{L}_{n-1}(x) \right], \quad n\geq 1,
\end{equation}
see e.g. \cite[Chapter 7, Problem 7.8]{Agarwal2009}. Moreover,
$\norm{\mathcal{L}_n}_{\infty} \leq 1$ for all $n\geq 0$.

The Legendre polynomials also form a complete orthogonal basis for the Lebesgue space $L^2(-1,1)$ \cite{Zeidler1995}, and satisfy the orthogonality condition
\begin{equation}
\label{E:LegOrthonormality}
\int_{-1}^{1} \mathcal{L}_n\, \mathcal{L}_m\,{\rm d}x = \frac{2\delta_{mn}}{2n+1},
\end{equation}
where $\delta_{mn}$ is the usual Kronecker delta.
This means that any square-integrable function $u$ can be expanded with a convergent series (in the $L^2$ norm sense)
\begin{equation}
\label{E:LegCoeffsIntegral}
u(x) = \sum_{n=0}^{\infty} \hat{u}_n \mathcal{L}_n(x),
\quad 
\hat{u}_n  = \int_{-1}^{1} u \,\mathcal{L}_n \,{\rm d}x,
\end{equation}
where the $\hat{u}_n$'s are known as Legendre coefficients. From~\eqref{E:LegOrthonormality} it follows that
\begin{equation}
\label{E:LegExpNorm}
\norm{u}_2^2 = \int_{-1}^{1} \vert u \vert^2 \,{\rm d}x = \sum_{n=0}^{\infty}\frac{2\vert
\hat{u}_n\vert^2}{2n+1}.
\end{equation}

Finally, if $u$ is continuously differentiable on $[-1,1]$ its Legendre series expansion converges uniformly. In fact, $u$ is Lipschitz on $[-1,1]$ because, by Taylor's theorem, for any $x,y\in [-1,1]$ there exists a point $z$ between $x$ and $y$ such that
$\vert u(y) - u(x) \vert  = \vert \partial u(z)\vert \,\vert x-y\vert
\leq C \,\vert x-y\vert$.
Here, $C$ is a generic positive constant whose existence is guaranteed by the continuity of $\partial u$ in $[-1,1]$. Uniform convergence follows from~\cite[Theorem XI and subsequent comments]{Jackson1930}.

\section{Proofs}

\subsection{Proof of Theorem~\ref{T:ConvergenceOuterSDPrelax}}
\label{A:ProofConvergenceOuterAprox}

Define the norm $\norm{\vec{w}}_k^2:=\int_{-1}^{1} (\mathcal{D}^k\vec{w})^T \mathcal{D}^k\vec{w} \dx$, consider the functional 
\begin{equation*}
\mathcal{H}_{\vec{\gamma}}\{\vec{w}\} := \frac{\mathcal{F}_{\vec{\gamma}}\{\vec{w}\}}{\norm{\vec{w}}_k^2},
\end{equation*}
and let
\begin{align*}
t(\vec{\gamma}) &\defeq 
\inf_{\vec{w}\in H\setminus\{\vec{0}\} } 
\mathcal{H}_{\vec{\gamma}}\{\vec{w}\},
\\
t_N(\vec{\gamma}) &\defeq 
\inf_{\vec{w}\in S_N\setminus\{\vec{0}\} }  
\mathcal{H}_{\vec{\gamma}}\{\vec{w}\}.
\end{align*}
(We need not assume that these infima are achieved.) It is not too difficult to show that the sets $T$ and $T^\mathrm{out}_N$ are described by the inequalities $t(\vec{\gamma})\geq 0$ and $t_N(\vec{\gamma})\geq 0$, respectively.  To prove Theorem~\ref{T:ConvergenceOuterSDPrelax} we rely on the following result.

\begin{lemma}
\label{T:LemmaMinSeqPoly}
Suppose $\vec{\gamma}\notin T$, i.e., there exists $\varepsilon_{\vec{\gamma}}>0$ such that $t(\vec{\gamma})\leq -2\varepsilon_{\vec{\gamma}}$. Then, there exists an integer $N_{\vec{\gamma}}$ such that $t_N(\vec{\gamma})\leq -\varepsilon_{\vec{\gamma}}$ for all $N\geq N_{\vec{\gamma}}$.
\end{lemma}

\begin{proof}
Let $\{\vec{w}_n\}_{n\geq 0}$, $\vec{w}_n\in H$, $\vec{w}_n \neq \vec{0}$ be a minimizing sequence, i.e. such that
$
\lim_{n\to\infty} 
\mathcal{H}_{\vec{\gamma}}\{\vec{w}_n\} = t(\vec{\gamma}),
$
and for each $n$ define $\mu_n := \norm{\vec{w}_n}_k^2/(n+1)$.
Note that $F\left(x,\mathcal{D}^{\vec{k}}\vec{w}(x)\right)$---the integrand of $\mathcal{F}_{\vec{\gamma}}\{\vec{w}\}$---and the product $(\mathcal{D}^{\vec{k}}\vec{w}(x))^T\mathcal{D}^{\vec{k}}\vec{w}(x)$
are continuous with respect to all entries of the vector $\mathcal{D}^{\vec{k}}\vec{w}(x)$ at each fixed $x\in[-1,1]$.
Using
\begin{equation*}
\left\vert 
\mathcal{F}_{\vec{\gamma}}
\{\vec{w}\} -
\mathcal{F}_{\vec{\gamma}}
\{\vec{w}_n\}
\right\vert 
\leq 
2 \norm{F\left(x,\mathcal{D}^{\vec{k}}\vec{w}\right) - F\left(x,\mathcal{D}^{\vec{k}}\vec{w}_n\right)}_\infty
\end{equation*} 
and a similar inequality for 
$\left\vert \norm{\vec{w}}_k^2 - \norm{\vec{w}_n}_k^2 \right\vert$
it is then not difficult to show that there exists $\delta_n>0$ such that
\begin{equation}
\label{E:DeltaCond}
\max_{0\leq\alpha\leq k} \norm{ \partial^\alpha \vec{w} - \partial^\alpha \vec{w}_n}_{\infty} \leq \delta_n
\end{equation}
implies
\begin{subequations}
\begin{align}
\label{E:Conditions1}
\left\vert \mathcal{F}_{\vec{\gamma}}\{\vec{w}\} -
\mathcal{F}_{\vec{\gamma}}\{\vec{w}_n\}\right\vert &\leq \mu_n,
\\
\label{E:Conditions2}
\left\vert \norm{\vec{w}}_k^2 - \norm{\vec{w}_n}_k^2 \right\vert &\leq \mu_n.
\end{align}
\end{subequations}

Since the Weierstrass approximation theorem can be extended to linear subspaces of continuously differentiable functions with prescribed boundary conditions (this follows e.g from~\cite[Proposition 2]{Peet2007}), there exists a polynomial $\vec{P}_{n}\in H$ of degree $d_n$, $\vec{P}_n \neq \vec{0},$ that satisfies~\eqref{E:DeltaCond}. Without loss of generality, we may assume that $d_n<d_{n+1}$. 
From~\eqref{E:Conditions1}--\eqref{E:Conditions2} we see that $\norm{\vec{P}_n}_k^2 \geq \norm{\vec{w}_n}_k^2 - \mu_n = \norm{\vec{w}_n}_k^2\times n/(n+1)$, and since $\vec{P}_{n}\in S_N$ for all $N\in\{d_n,\,\hdots,\,d_{n+1}-1\}$  we can write
\begin{equation*}
\begin{aligned}
t(\vec{\gamma}) \leq t_N(\vec{\gamma}) 
&\leq  \mathcal{H}_{\vec{\gamma}}\{\vec{P}_{n}\}
\\&
\leq 
\frac{\vert \mathcal{F}_{\vec{\gamma}}\{\vec{P}_{n}\} - \mathcal{F}_{\vec{\gamma}}\{\vec{w}_n\}\vert}
{\min\{\norm{\vec{P}_n}_k^2,\,\norm{\vec{w}_n}_k^2\}}
 + \mathcal{H}_{\vec{\gamma}}\{\vec{w}_n\}
\\&
\leq \frac{1}{n} + \mathcal{H}_{\vec{\gamma}}\{\vec{w}_n\}.
\end{aligned}
\end{equation*}

The last expression tends to $t(\vec{\gamma})$ as $n$ (hence, $N$) tends to infinity, so $t_N(\vec{\gamma}) - t(\vec{\gamma}) \downarrow 0$ as $N\to\infty$. In particular, there exists an integer $N_{\vec{\gamma}}$ such that $t_N(\vec{\gamma}) - t(\vec{\gamma})\leq \varepsilon_{\vec{\gamma}}$ for all $N\geq N_{\vec{\gamma}}$. Upon rearranging and recalling that $t(\vec{\gamma})\leq -2\varepsilon_{\vec{\gamma}}$ we conclude that, for all $N\geq N_{\vec{\gamma}}$,
\begin{equation*}
t_N(\vec{\gamma}) \leq \varepsilon_{\vec{\gamma}} + t(\vec{\gamma}) \leq \varepsilon_{\vec{\gamma}}-2\varepsilon_{\vec{\gamma}}  = - \varepsilon_{\vec{\gamma}}< 0.
\qedhere
\end{equation*}
\end{proof}

Let us now prove Theorem~\ref{T:ConvergenceOuterSDPrelax}. The sequence of optimal values $\{p^\star_N\}_{N\geq 0}$ is non-decreasing since $T^\mathrm{out}_{N+1} \subset T^\mathrm{out}_N$.
To prove convergence when~\eqref{E:OptProblem} achieves its optimal value, let us assume that its feasible set $T$ is {\it bounded}; if not, one can formulate an equivalent problem (meaning that $\vec{\gamma}^\star$ is still an optimal solution) with bounded feasible set by adding the constraint $\norm{\vec{\gamma}}\leq r$ for a sufficiently large $r>0$. For any $\varepsilon> 0$ (different from that used in Lemma~\ref{T:LemmaMinSeqPoly}), the set
\begin{equation*}
K := \{\vec{\gamma} : \, \varepsilon \leq {\rm dist} (\vec{\gamma},T) \leq 2\,\varepsilon \},
\end{equation*}
 where 
$ {\rm dist} (\vec{\gamma},T) 
= \min_{ \vec{\eta} \in T} 
\norm{ \vec{\eta} - \vec{\gamma} }$ 
is the usual euclidean distance of $\vec{\gamma}$ from $T$, is compact,  and only contains  points that are infeasible for~\eqref{E:OptProblem}. 
By Lemma~\ref{T:LemmaMinSeqPoly}, for each $\vec{\gamma}\in K$ there exists an integer $N_{\vec{\gamma}}$ such that $t_{N}(\vec{\gamma})<0$, i.e. $\vec{\gamma}$ is infeasible for~\eqref{E:OuterSDP}, for all $N\geq N_{\vec{\gamma}}$. 
The compactness of $K$ and the continuity of $t_N(\vec{\gamma})$---the proof of this fact is not difficult and is left to the reader---imply the existence of an integer $N_0=N_0(\varepsilon)$ and a finite number of balls $B(\vec{\gamma}_i,\delta_i)$ with center $\vec{\gamma}_i$ and radius $\delta_i$ which cover $K$ such that $t_{N}(\vec{\gamma})<0$ in each ball for all $N\geq N_0$. Consequently, all points in $K$ are infeasible for the outer SDP relaxation~\eqref{E:OuterSDP} when $N\geq N_0$. Since the feasible set $T^\mathrm{out}_N$ of the outer SDP relaxation must be convex, we conclude that it must be contained within an $\varepsilon$-neighbourhood of $T$ for all $N\geq N_0$, {\it i.e.},
\begin{equation*}
\forall N\geq N_0, \quad \max_{ \vec{\gamma}\in T^\mathrm{out}_N } {\rm dist} (\vec{\gamma},T) < \varepsilon.
\end{equation*}
In particular, $T^\mathrm{out}_N$ is bounded, and there exists a point $\vec{\gamma}^\star_N$ with $\vec{c}^T\vec{\gamma}_N^\star = p^\star_N$ whose projection onto $T$, denoted $\mathcal{P}_T(\vec{\gamma}_N^\star)$, satisfies $\norm{\mathcal{P}_T(\vec{\gamma}_N^\star) - \vec{\gamma}_N^\star} < \varepsilon$. Then, for all $N\geq N_0$
\begin{align*}
p^\star  -   \norm{\vec{c}} \varepsilon
&\leq \vec{c}^T \mathcal{P}_T(\vec{\gamma}_N^\star)  
-  \norm{\vec{c}} \varepsilon
\\
&= \vec{c}^T \left[ \mathcal{P}_T(\vec{\gamma}_N^\star) - \vec{\gamma}_N^\star \right] + p^\star_N -  \norm{\vec{c}} \varepsilon
\\
&\leq \norm{\vec{c}}\,\norm{ \mathcal{P}_T(\vec{\gamma}_N^\star) - \vec{\gamma}_N^\star} + p^\star_N -  \norm{\vec{c}} \varepsilon
\\
&< p^\star_N.
\end{align*} 
Since $p_N^\star\leq p^\star$ we conclude that $p^\star  -   \norm{\vec{c}} \varepsilon \leq \lim_{N\to\infty} p_N^\star \leq p^\star$ for any $\varepsilon$, and the proof is concluded by letting $\varepsilon\to 0$.

\subsection{Proof of Lemma~\ref{T:IntegrationLemma}}
\label{A:ProofIntegrationLemma}

The statement is trivial when $\alpha=k$. Moreover, since $u\in C^{m}([-1,1])$ and $k\leq m-1$, the 
Legendre expansions of all derivatives $\der{\alpha}u$, $\alpha\in\{0,...,k\}$ 
converge uniformly, cf. Appendix~\ref{A:LegPolys}. Consequently, we can use the fundamental 
theorem of calculus for each $\alpha\leq k-1$ to write
\begin{equation}
\label{E:Theorem1DiffFormulaProof}
\left(\der{\alpha} u\right)(x) 
= \der{\alpha}u\at{-1} + \int_{-1}^{x} 
\partial^{\alpha+1} u(t) \,{\rm d}t
=
\partial^\alpha u\at{-1} +
\sum_{n\geq 0} \hat{u}^{\alpha+1}_n \, \int_{-1}^{x}  \mathcal{L}_n(t) \,{\rm d}t.
\end{equation} 
The last expression can be integrated recalling that $\mathcal{L}_0(x)=1$, 
$\mathcal{L}_1(x)=x$, $\mathcal{L}_n(\pm1) = (\pm 1)^n$ and using the recurrence 
relation~\eqref{E:LegRecursionDerivatives}. We can then rewrite~\eqref{E:Theorem1DiffFormulaProof} as
\begin{equation*}
\partial^\alpha u
=\partial^\alpha u\at{-1} +  \left[\mathcal{L}_1\!+\!\mathcal{L}_0 \right] 
\hat{u}^{\alpha+1}_0 
+ \sum_{n\geq 1} \left[\mathcal{L}_{n+1} - \mathcal{L}_{n-1} \right] 
\frac{\hat{u}^{\alpha+1}_n}{2n+1}.
\end{equation*}
Rearranging the series and comparing coefficients with the Legendre expansion of 
$\der{\alpha}u$ gives the relations
\begin{subequations}
\label{E:DifferentiationRule}
\begin{align}
\label{E:DifferentiationRule_0}
\hat{u}^{\alpha}_0 &= \partial^\alpha u\at{-1} + \hat{u}^{\alpha+1}_0 - 
\frac{1}{3}\hat{u}^{\alpha+1}_1,
\\
\label{E:DifferentiationRule_others}
\hat{u}^{\alpha}_n &= \frac{\hat{u}^{\alpha+1}_{n-1}}{2n-1} - 
\frac{\hat{u}^{\alpha+1}_{n+1}}{2n+3}, \qquad\qquad\qquad n\geq 1.
\end{align}
\end{subequations}

We can then find matrices $\vec{C}^{\alpha}$ and $\vec{E}^{\alpha}$ such that
\begin{equation}
\label{E:FirstIter}
\begin{aligned}
\vec{\hat{u}}^\alpha_{[r,s]}
&= \vec{E}^\alpha \mathcal{D}^{k-1}u\at{-1} + \vec{C}^\alpha \vec{\hat{u}}^{\alpha+1}_{[r-1,s+1]}.
\end{aligned}
\end{equation}
Here and in the following, it should be understood that negative indices should be replaced by $0$. Before proceeding, note that strictly speaking the matrices $\vec{C}^{\alpha}$ and $\vec{E}^{\alpha}$ depend on $r$ and $s$, but we do not write this explicitly to ease the notation. In particular,~\eqref{E:DifferentiationRule} implies that $\vec{E}^\alpha=\vec{0}$ if $r\geq 1$.

Expressions similar to~\eqref{E:FirstIter} can be built for all vectors $\vec{\hat{u}}^{\alpha+i}_{[r-i,s+i]}$, $i \in\{0,\,\hdots,\,k-\alpha-1\}$. After some algebra, it is therefore possible to write
\begin{equation}
\label{E:IntegrationLemmaEq1}
\vec{\hat{u}}^\alpha_{[r,s]} 
=\vec{B}^\alpha_{[r,s]} 
\mathcal{D}^{k-1}u\at{-1} 
+ \left(\prod_{i=0}^{k-\alpha-1}\!\!\vec{C}^{\alpha+i}\right)\!\vec{\hat{u}}^k_{[r-k+\alpha,s+k-\alpha]},
\end{equation}
where
\begin{equation*}
\vec{B}^\alpha_{[r,s]} \defeq  \vec{E}^\alpha + \vec{C}^\alpha\vec{E}^{\alpha+1}+ 
\cdots + 
\left(\prod_{i=0}^{k-\alpha-2}\vec{C}^{\alpha+i}\right)\vec{E}^{k-1}.
\end{equation*}
Note that, in light of~\eqref{E:DifferentiationRule}, all matrices $\vec{E}^{\alpha+i}$, $i \in\{0,\,\hdots,\,k-\alpha-1\}$, are zero if $r \geq k-\alpha$.
Since we have assumed that $s+k-\alpha \leq M$, the last term in~\eqref{E:IntegrationLemmaEq1} can be rewritten in terms of $\vec{\hat{u}}^k_{[0,M]}$ (recall that $r-k+\alpha$ is replaced by $0$ if it is negative). The proof is concluded by defining 
\begin{align}
\label{E:DrsDef}
\vec{D}^\alpha_{[r,s]} &\defeq \begin{bmatrix}
\vec{0}_{(s-r+1)\times(r-k+\alpha)}, && \displaystyle\prod_{i=0}^{k-\alpha-1}\vec{C}^{\alpha+i}, && \vec{0}_{(s-r+1)\times(M-s-k+\alpha)}
\end{bmatrix},
\end{align}
where the size of the zero matrices is indicated by subscripts.

\subsection{Proof of Lemma~\ref{T:BCLemma}}
\label{A:ProofBCLemma}

Recalling the definition of $\mathcal{B}^{k-1}u$, we only need to show that $\mathcal{D}^{k-1}u\at{1}$ can be expressed
as linear combination of the entries of $\vec{\check{u}}_M$. 
Applying the fundamental theorem of calculus as in Appendix~\ref{A:ProofIntegrationLemma}, it may be shown that
$\der{\alpha}u\at{1} 
=\der{\alpha}u\at{-1} + 2\hat{u}^{\alpha+1}_0$
for any $\alpha\in\{0,\,\hdots,\,k-1\}$. By Lemma~\ref{T:IntegrationLemma}, $\der{\alpha}u\at{1}$
can then be written as a linear combination of the entries of $\vec{\check{u}}_M$. Repeating this argument for all $\alpha\in\{0,\,\hdots,\,k-1\}$, we conclude the same for all entries of
$\mathcal{D}^{k-1}u\at{1}$, proving the existence of $\vec{G}_M$.

\subsection{Proof of Lemma~\ref{T:LemmaQ}}
\label{A:ProofLemmaQ}

\emph{(i)} Recall~\eqref{E:RemFunDef} and expand
\begin{equation*}
\mathcal{Q}^{\alpha\beta}_{uv} = \sum_{m=0}^{N_\alpha}\sum_{n=N_\beta+1}^{+\infty} \hat{u}^\alpha_m \hat{v}^\beta_n \int_{-1}^{1} f\,\Leg_m\,\Leg_n\,{\rm d}x 
+ \sum_{m=N_\alpha+1}^{\infty}\sum_{n=0}^{N_\beta} \hat{u}^\alpha_m \hat{v}^\beta_n \int_{-1}^{1} f\,\Leg_m\,\Leg_n\,{\rm d}x,
\end{equation*} 
where $N_\alpha = N+\alpha$ and $N_\beta = N+\beta$. Since $f$ is a polynomial of degree at most $d_F$, the product $f\mathcal{L}_m$ is a polynomial of degree at most $m+d_F$, so it is orthogonal to any Legendre polynomial $\mathcal{L}_n$ with $n>m+d_F$. In particular, it may be shown~\cite{Dougall1953}  that the integral $\int_{-1}^{1}f\Leg_n\Leg_m \dx$ vanishes if $\vert m-n\vert > d_F$. Using the short-hand notation $\overline{n}=n+1-d_F$, we can write 
\begin{equation}
\label{E:Qexpansion}
\mathcal{Q}^{\alpha\beta}_{uv} =
\begin{bmatrix}{\hat{u}}^\alpha_{\overline{N_\beta}} \\ \vdots \\ 
\hat{u}^\alpha_{N_\alpha}\end{bmatrix}^T
\vec{\Phi}{}_{[\overline{N_\beta},N_\alpha]}^{[N_\beta+1,N_\alpha+d_F]}
\begin{bmatrix}\hat{v}^\beta_{N_\beta+1} \\ \vdots \\ 
\hat{v}^\beta_{N_\alpha+d_F}\end{bmatrix}
+
\begin{bmatrix}{\hat{v}}^\beta_{\overline{N_\alpha}} \\ \vdots \\ 
\hat{v}^\beta_{N_\beta}\end{bmatrix}^T
\vec{\Phi}{}_{[\overline{N_\alpha},N_\beta]}^{[N_\alpha+1,N_\beta+d_F]}
\begin{bmatrix}\hat{u}^\alpha_{N_\alpha+1} \\ \vdots \\ 
\hat{u}^\alpha_{N_\beta+d_F}\end{bmatrix}.
\end{equation}
Note that we have assumed that $\alpha$, $\beta$ and $d_F$ are such that $1-d_F \leq 
\alpha-\beta \leq d_F - 1$,
so that the vectors in~\eqref{E:Qexpansion} are well-defined. If the left 
(resp. right) inequality is not satisfied, then the first 
(resp. second) term in~\eqref{E:Qexpansion} vanishes. 
Since $N_\alpha+d_F\leq M+\beta-k$ and $N_\beta+d_F\leq M+\alpha-k$, we can apply Lemma~\ref{T:IntegrationLemma}, and our assumption that $N \geq d_F + k -1$ guarantees that $\overline{N_\alpha}\geq k-\beta$ and $\overline{N_\beta}\geq k-\alpha$, so there is no dependence on the boundary values. Consequently, we can find a matrix $\vec{Q}(\vec{\gamma})$ such that
\begin{equation}
\label{E:QuvabQuadForm}
\mathcal{Q}^{\alpha\beta}_{uv} = \left(\vec{\hat{u}}^k_{[0,M]}\right)^T
\, \vec{Q}(\vec{\gamma}) \, \vec{\hat{v}}^k_{[0,M]}.
\end{equation} 
The matrix $\vec{Q}^{\alpha\beta}_{uv}$ is found using~\eqref{E:DeflMatDef} after taking the symmetric part of the right-hand side of~\eqref{E:QuvabQuadForm}.

\emph{(ii)} Let
$
\vec{\nu}=\begin{bmatrix}
\hat{u}^k_{M+1},\, \hdots,\, \hat{u}^k_{M+d_F},\,
\hat{v}^k_{M+1},\, \hdots,\, \hat{v}^k_{M+d_F}\end{bmatrix}^T.
$
After replacing $N_\alpha$ and $N_\beta$ with $M$ in~\eqref{E:Qexpansion}, it may be verified using~\eqref{E:DeflMatDef} that
$
\mathcal{Q}^{kk}_{uv} = 2\,{\vec{\psi}_M}^T \,{\vec{L}_{\overline{M}}}^T \vec{Y} \vec{\nu}.
$
By~\eqref{E:AuxiliaryLMI},
\begin{align*}
0 &\leq \begin{bmatrix}
\vec{L}_{\overline{M}} {\vec{\psi}_M}\\ \vec{\nu}
\end{bmatrix}^T
\begin{bmatrix}
\vec{Q}^{kk}_{uv} & \vec{Y} \\
\vec{Y}^T & \vec{\Sigma}^{kk}_{uv}\otimes\vec{\Delta}
\end{bmatrix}
\begin{bmatrix}
\vec{L}_{\overline{M}} {\vec{\psi}_M}\\ \vec{\nu}
\end{bmatrix} 
\notag \\
&= 
{\vec{\psi}_M}^T \left({\vec{L}_{\overline{M}}}^T\vec{Q}^{kk}_{uv} \vec{L}_{\overline{M}}\right) {\vec{\psi}_M}
+ \vec{\nu}^T \left(\vec{\Sigma}^{kk}_{uv}\otimes\vec{\Delta}\right) \vec{\nu} + \mathcal{Q}^{kk}_{uv}.
\end{align*}
Now, $\vec{\Sigma}^{kk}_{uv}$ is a diagonal matrix by assumption. Recalling the definition of $\vec{\Delta}$, $\vec{\nu}$, and rearranging we arrive at
\begin{equation}
\label{E:LowerBoundQkk}
\mathcal{Q}^{kk}_{uv} 
\geq 
-{\vec{\psi}_M}^T \left({\vec{L}_{\overline{M}}}^T\vec{Q}^{kk}_{uv} \vec{L}_{\overline{M}}\right) {\vec{\psi}_M}
-(\vec{\Sigma}^{kk}_{uv})_{11} 
\sum_{n=M+1}^{M+d_F}\frac{2\vert\hat{u}^k_n\vert^2}{2n+1} 
-(\vec{\Sigma}^{kk}_{uv})_{22} 
\sum_{n=M+1}^{M+d_F}\frac{2\vert\hat{v}^k_n\vert^2}{2n+1}.
\end{equation}
Recognizing from~\eqref{E:LegExpNorm} that the two sums in~\eqref{E:LowerBoundQkk} can be bounded
by $\norm{U^k_M}_2^2$ and $\norm{V^K_M}_2^2$ we obtain~\eqref{E:LemmaQEstimate}.

\subsection{Proof of Lemma~\ref{T:LemmaR}}
\label{A:ProofLemmaR}

We start by determining an upper bound on $\|U^\alpha_{N_\alpha}\|^2_{2}$ in 
terms of the vector $\vec{\hat{u}}^k_{[0,M]}$ and $\|U^k_M\|^2_{2}$ (similar bounds can be found for 
$V^\beta_{N_\beta}$).  Recalling~\eqref{E:RemFunDef},~\eqref{E:Ncondition} and~\eqref{E:Mdef}, we can write
\begin{align}
\label{E:EstimateNorm1}
\frac{1}{2}\norm{U^\alpha_{N_\alpha}}^2_{2} 
&= \sum_{n=N_\alpha+1}^{M-k+\alpha} \frac{\left(\hat{u}^\alpha_n\right)^2}{2n+1} 
+ \sum_{n=M-k+\alpha+1}^{+\infty} \frac{\left(\hat{u}^\alpha_n\right)^2}{2n+1}
\notag \\
&= \left(\vec{\hat{u}}^k_{[0,M]}\right)^T \,\vec{H}_\alpha \,\vec{\hat{u}}^k_{[0,M]} +\sum_{n=M-k+\alpha+1}^{+\infty} 
\frac{\left(\hat{u}^\alpha_n\right)^2}{2n+1},
\end{align}
where the matrix $\vec{H}_\alpha$ can be obtained from
Lemma~\ref{T:IntegrationLemma}. In particular, we note that~\eqref{E:DifferentiationRule_others} is applied $k-a$ times to $(\hat{u}^\alpha_n)^2$ to compute $\vec{H}_\alpha$, and since $n > N_\alpha \geq N$ it may be verified that $\norm{\vec{H}_\alpha}_F\sim N^{-2(k-\alpha)-1}$.

When $\alpha=k$, the last term in~\eqref{E:EstimateNorm1} is $\norm{U^k_M}_2^2/2$, so
\begin{equation}
\label{E:AlphaEqK}
\frac{1}{2}\norm{U^k_{N_k}}^2_{2} = 
\left( \vec{\hat{u}}^k_{[0,M]} \right)^T\, \vec{H}_k \,\vec{\hat{u}}^k_{[0,M]}
+
\frac{1}{2}\norm{U^k_M}_2^2.
\end{equation}

When $\alpha\leq k-1$, instead, we define
\begin{equation}
\label{E:OmegaDef}
\omega_{\eta} = \frac{4}{[2(M-k+\eta)+1][2(M-k+\eta)+5]}
\end{equation}
for $\eta \in \{0,\,\hdots,\,k-1\}$ and use~\eqref{E:DifferentiationRule}, the 
elementary inequality $(a-b)^2\leq 2(a^2+b^2)$, and appropriate changes of indices to show 
\begin{align}
\label{E:EstimateSumTerm}
\sum_{n=M-k+\alpha+1}^{+\infty} 
\frac{ \left(\hat{u}^{\alpha}_{n}\right)^2 }{ 2n+1 }
&\leq 
\sum_{n=M-k+\alpha+1}^{+\infty} 
\frac{2}{2n+1} \left[ 
\frac{\vert\hat{u}^{\alpha+1}_{n-1}\vert^2}{(2n-1)^2} 
+ \frac{\vert\hat{u}^{\alpha+1}_{n+1}\vert^2}{(2n+3)^2}
\right]
\notag\\
&\leq
\sum_{n=M-k+\alpha}^{M-k+\alpha+1} 
\frac{2\vert\hat{u}^{\alpha+1}_{n}\vert^2}{(2n+3)(2n+1)^2} 
+
\sum_{n=M-k+\alpha+2}^{\infty} 
\frac{4\vert\hat{u}^{\alpha+1}_{n}\vert^2}{(2n-1)(2n+1)(2n+3)} 
\notag\\
&\leq
\sum_{n=M-k+\alpha}^{M-k+\alpha+1} 
\frac{2\vert\hat{u}^{\alpha+1}_{n}\vert^2}{(2n+3)(2n+1)^2} 
+\omega_{\alpha+1} \sum_{n=M-k+\alpha+2}^{+\infty} 
\frac{\vert\hat{u}^{\alpha+1}_{n}\vert^2}{2n+1}.
\end{align}
Applying Lemma~\ref{T:IntegrationLemma} to the first term on the 
right-hand side of~\eqref{E:EstimateSumTerm} and substituting back into~\eqref{E:EstimateNorm1}, we 
can construct a matrix $\vec{T}_\alpha$ such that
\begin{equation}
\label{E:Estimate1AppendixC}
\frac{1}{2}\norm{U^\alpha_{N_\alpha}}^2_{2} 
\leq 
\left( \vec{\hat{u}}^k_{[0,M]} \right)^T\, \vec{T}_{\alpha}
\,\vec{\hat{u}}^k_{[0,M]}
+ \omega_{\alpha+1}\sum_{n=M-k+\alpha+2}^{+\infty} 
\frac{\vert\hat{u}^{\alpha+1}_{n}\vert^2}{2n+1}.
\end{equation}
As for $\vec{H}_\alpha$, it may be verified that $\norm{\vec{T}_\alpha}_F\sim N^{-2(k-\alpha)-1}$.

Similar estimates can be carried out for the infinite sum on the right-hand side 
of~\eqref{E:Estimate1AppendixC}. By recursion, we can eventually construct a matrix 
$\vec{Z}_\alpha$ and a constant $\lambda_\alpha$ such that
\begin{equation}
\label{E:FinalEstimate_utilde}
\begin{aligned}
\frac{1}{2}\norm{U^\alpha_{N_\alpha}}^2_{2}
\leq \left( \vec{\hat{u}}^k_{[0,M]} \right)^T \, \vec{Z}_\alpha
\, \vec{\hat{u}}^k_{[0,M]}
+\lambda_\alpha
\norm{U^k_M}^2_{2}.
\end{aligned}
\end{equation}
Note that $\norm{\vec{Z}_\alpha }_F \sim {N^{-2(k-\alpha)-1}}$, while $\lambda_\alpha \sim N^{-2(k-\alpha)}$ since every recursion step introduces a 
factor of $N^{-2}$ according to~\eqref{E:OmegaDef}. Moreover, the right-hand side of~\eqref{E:FinalEstimate_utilde} has the same form 
as~\eqref{E:AlphaEqK}, so for the rest of this section we will not distinguish the cases $\alpha\leq k-1$ and 
$\alpha=k$.

The estimate~\eqref{E:FinalEstimate_utilde} can be used in conjunction with Young's inequality and~\eqref{E:DeflMatDef} to show that for any $\varepsilon>0$ we can bound
\begin{equation*}
\vert \mathcal{R}^{\alpha\beta}_{uv} \vert
\leq 
\norm{f}_\infty 
{\vec{\psi}_M}^T \left(
 {\vec{L}_0}^T
\begin{bmatrix}
\varepsilon \vec{Z}_\alpha  & \vec{0} \\ 
\vec{0} & \frac{1}{\varepsilon}
\vec{Z}_\beta 
\end{bmatrix}
\vec{L}_0 
\right) {\vec{\psi}_M}^T
+ \norm{f}_\infty 
\left( \varepsilon\lambda_\alpha\norm{U_k}^2_{2} + 
\frac{\lambda_\beta }{\varepsilon}
\norm{V_k}^2_{2}\right).
\end{equation*}
We now set $\varepsilon=(N+1)^{\beta-\alpha}$, so
$\varepsilon\lambda_\alpha \sim \varepsilon^{-1}\lambda_\beta \sim 
{N^{\alpha+\beta-2k}}$ and
$\norm{\varepsilon\vec{Z}_\alpha}_F \sim \norm{\varepsilon^{-1}\vec{Z}_\beta}_F 
\sim N^{\alpha+\beta-2k-1}$, and let
\begin{align*}
\vec{R}^{\alpha\beta}_{uv} &\defeq	
{\vec{L}_0}^T
\begin{bmatrix}
\varepsilon \vec{Z}_\alpha  & \vec{0} \\ 
\vec{0} & \frac{1}{\varepsilon}
\vec{Z}_\beta 
\end{bmatrix}
\vec{L}_0 ,
&
\vec{\Sigma}^{\alpha\beta}_{uv} &\defeq
\begin{bmatrix}
\varepsilon\lambda_\alpha & 0 \\
0 &\varepsilon^{-1}\lambda_\beta
\end{bmatrix}.
\end{align*}
Recalling that $\left\vert\mathcal{L}_n(x) \right\vert\leq 1$ for all $n\geq 0$~\cite{Jackson1930}, equation~\eqref{E:Restimate} then follows from the estimate
\begin{equation*}
\norm{f}_\infty 
=\sup_{x\in [-1,1]} 
\left\vert \sum_{n=0}^{p} \hat{f}_n(\gamma) \mathcal{L}_n(x) \right\vert
\leq  \sum_{n=0}^{p} \left\vert\hat{f}_n(\gamma)\right\vert 
= \norm{\vec{\hat{f}}(\gamma)}_1.
\end{equation*}

\bibliographystyle{plain}
\bibliography{references}

\end{document}